\documentclass[12pt]{article}
\usepackage[english]{babel}
\usepackage{amsthm,amsfonts, amsbsy, amssymb,amsmath,graphicx}
\usepackage{graphics}
\usepackage[all]{xy}

\newtheorem{theorem}{Theorem}

\newtheorem{corollary}{Corollary}

\newtheorem{proposition}{Proposition}
\theoremstyle{definition}
\newtheorem{definition}{Definition}
\newtheorem{example}{Example}
\theoremstyle{remark}
\newtheorem{remark}{Remark}

\def\Z{{\mathbb Z}}

\def\R{{\mathbb R}}
\def\X{{\mathcal X}}

\DeclareMathOperator{\sign}{sign}
\DeclareMathOperator{\lk}{lk}
\DeclareMathOperator{\im}{im}
\DeclareMathOperator{\id}{id}
\date{}

\author{Igor Nikonov\footnote{Lomonosov Moscow State University}}
\title{Virtual index cocycles and invariants of virtual links}
\date{}

\begin{document}

\maketitle

\begin{abstract}
Virtual index cocycle is the 1-cochain that counts virtual crossings in the arcs of a virtual link diagram. We show how this cocycle can be used to reformulate and unify some known invariants of virtual links.
\end{abstract}

\section{Introduction}

A conceptional way to define virtual links uses {\em virtual diagrams}. Those are plane 4-valent graphs with two types of vertices: classical and virtual ones, see Fig.~\ref{fig:classical_virtual_crossings}.

\begin{figure}[h]
\centering\includegraphics[width=0.25\textwidth]{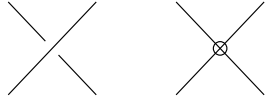}
\caption{A classical (left) and a virtual (right) crossings}\label{fig:classical_virtual_crossings}
\end{figure}

The diagrams undergo classical and virtual  (Fig.~\ref{fig:reidemeister_moves}) Reidemeister moves. The moves induce an equivalence on the set of virtual diagrams and the equivalence classes are called {\em virtual knots and links}.

\begin{figure}[h]
\centering\includegraphics[width=0.25\textwidth]{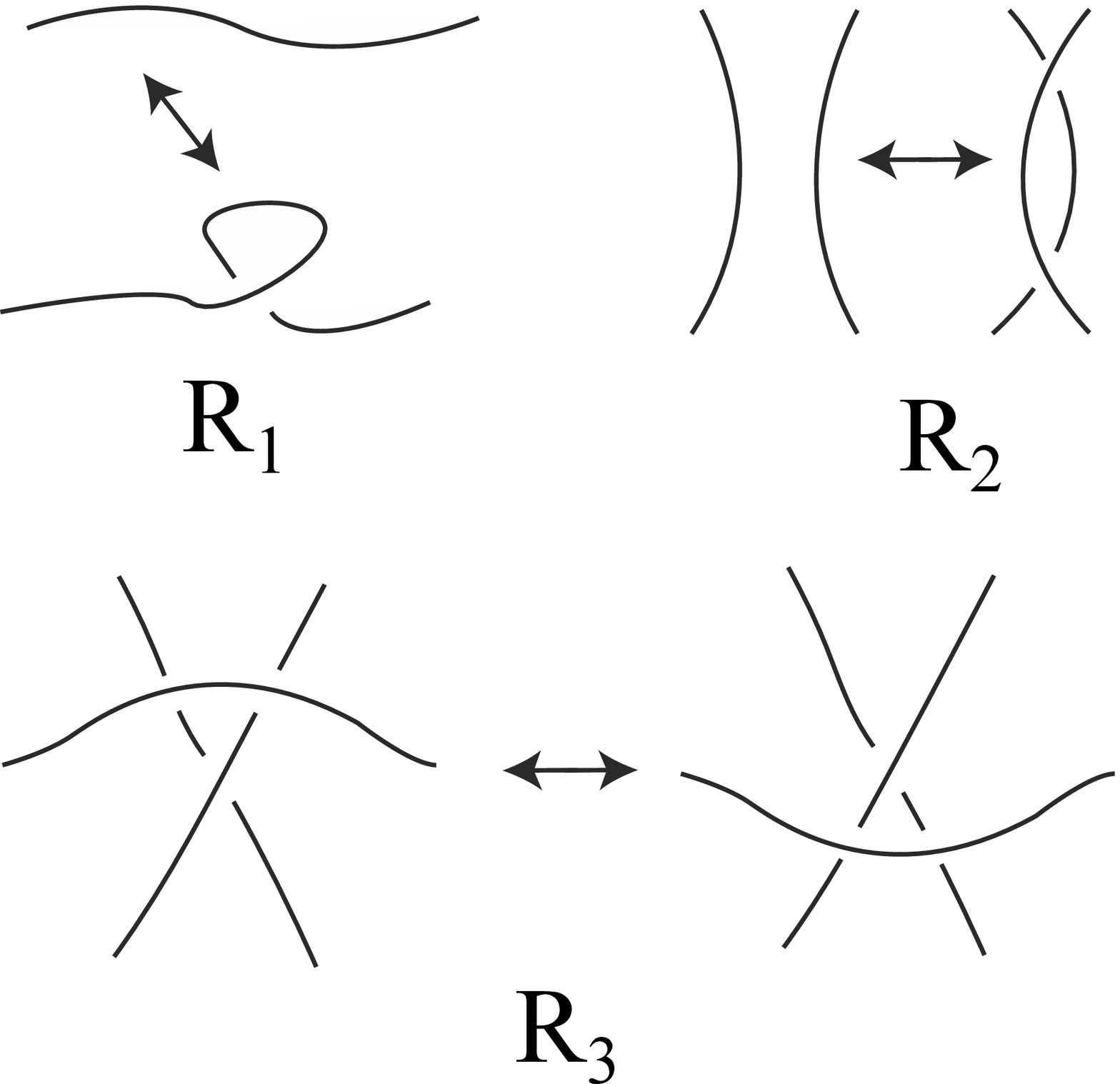}\qquad
\includegraphics[width=0.28\textwidth]{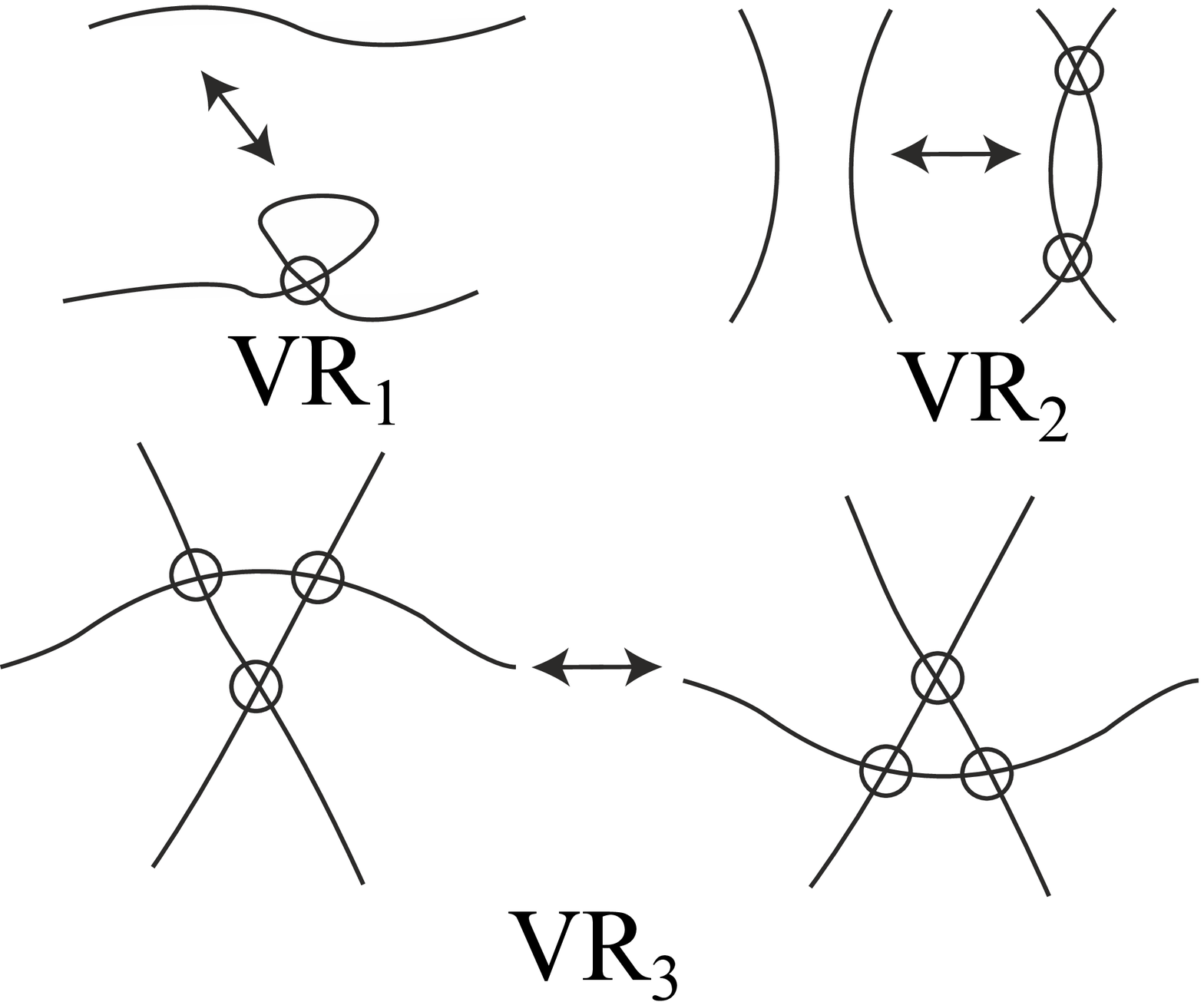}\\
\includegraphics[width=0.25\textwidth]{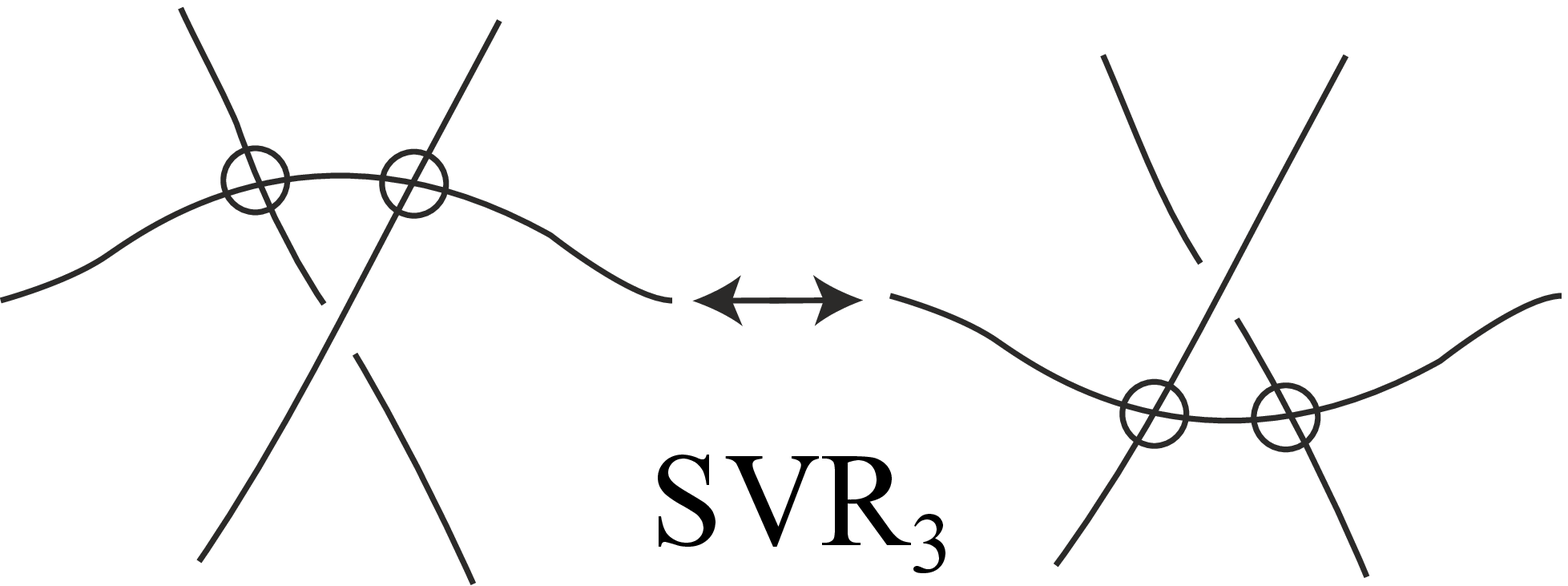}
\caption{Classical ($R_1-R_3$) and virtual ($VR_1-SVR_3$) Reidemeister moves}\label{fig:reidemeister_moves}
\end{figure}

Many invariants of virtual knots are defined as computable properties of virtual diagrams which are preserved under Reidemester moves. Some invariants ignore virtual crossings (e.g., Jones polynomial or the knot group), but for others, virtual crossing are an essential ingredient of the construction. Among such invariants are virtual intersection index polynomials~\cite{IKP,ILL,Im2013}, VA-polynomial~\cite{Manturov2002,Manturov2003}, virtual (bi)quandle~\cite{KauffmanManturov2005,Manturov2002} etc.

On the other hand, virtual knots can be defined by means Gauss diagrams. The {\em Gauss diagram} $G=G(D)$ of a virtual knot diagram $D$ is a chord diagram whose chords correspond to the classical crossings of $D$, see Fig.~\ref{fig:virtual_gauss_diagram}. The chords carry orientation (from over-crossing to under-crossing) and the sign of the crossings, see Fig.~\ref{fig:crossing_sign}.

\begin{figure}[h]
\centering\includegraphics[width=0.15\textwidth]{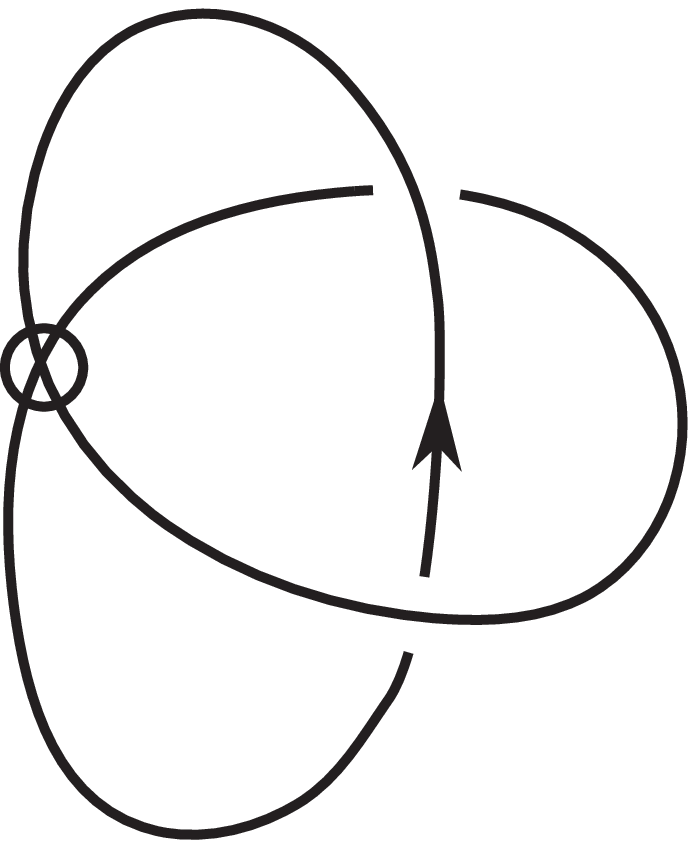}\qquad
\includegraphics[width=0.15\textwidth]{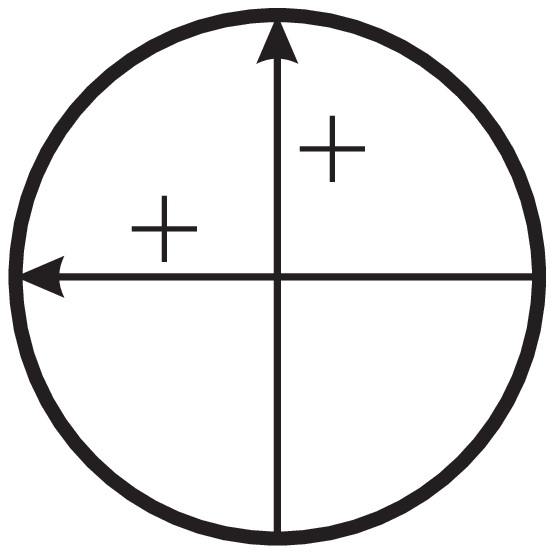}
\caption{A virtual knot and its Gauss diagram}\label{fig:virtual_gauss_diagram}
\end{figure}

\begin{figure}[h]
\centering\includegraphics[width=0.2\textwidth]{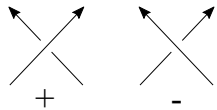}
\caption{The sign of a crossing}\label{fig:crossing_sign}
\end{figure}

Classical Reidemeister moves induce transformations of Gauss diagrams (see Fig.~\ref{fig:reidemeister_gauss}), and virtual Reidemeister moves have no effect on the diagrams. Virtual knots are exactly the equivalence classes of Gauss diagrams modulo the induced Reidemeister moves.

\begin{figure}[h]
\centering\includegraphics[width=0.6\textwidth]{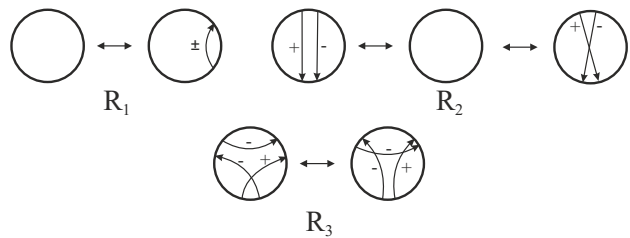}
\caption{Reidemeister moves on Gauss diagrams}\label{fig:reidemeister_gauss}
\end{figure}

There are no virtual crossings in Gauss diagrams. This means that any virtual knot invariant has no need of virtual crossings. If invariant's construction uses information on them, it can be extracted and replaced with something concerning only classical crossings. Performing this procedure for some knot invariants is the aim of this article.

The paper is organized as follows. In Section~\ref{sect:virtual_index_cocycle} we introduce virtual crossing cocycle that counts the algebraical number of virtual crossings on the arcs of a virtual diagram and show how this cocycle can be expressed in term of the index in the sense of~\cite{Cheng2013}. The subsequent section we consider how the cocycle can be used for reformulation of some invariants of virtual links in a form which does not account virtual crossing. We conclude the paper with a modification of the construction of Khovanov homology for virtual links~\cite{DKK2017,Manturov2007}; the modification exploits the parity cocycle, i.e. the index cocycle mod $2$.

The author was supported by the Russian Foundation for Basic Research (grant No. 19-01-00775-a).

The author is grateful to Louis Kauffman and Eiji Ogasa for fruitful discussions.

\section{Virtual index cochain}\label{sect:virtual_index_cocycle}

Let $D$ be an oriented virtual link diagram. It can be considered as an immersion of some 4-valent {\em covering graph} $\tilde D$ into the plane: the classical crossings of $D$ are the images of the vertices of $\tilde D$ and the virtual crossings are self-intersection which appear by the immersion.

Let $G(D)$ be the Gauss diagram of $D$. The graph $\tilde D$ can be obtained from $G(D)$ by collapsing each chord of $G(D)$ to a point. Thus, we have a sequence of projections $G(D)\stackrel{p_1}{\rightarrow}\tilde D \stackrel{p_2}{\rightarrow}D$, see Fig.~\ref{fig:gauss_virtual_diagrams}. 

\begin{figure}[h]
\centering\includegraphics[width=0.5\textwidth]{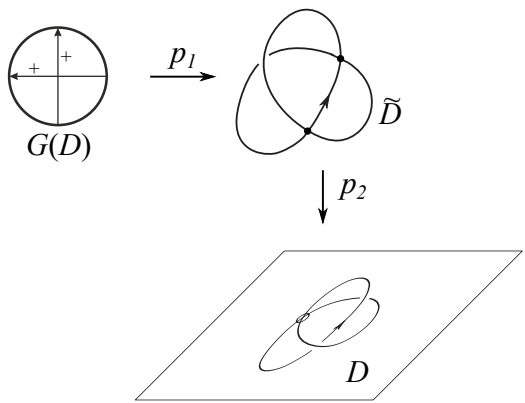}
\caption{A Gauss diagram, a graph and a virtual diagram}\label{fig:gauss_virtual_diagrams}
\end{figure}

The edges of the graph $\tilde D$ correspond to the {\em long arcs} of the diagram $D$, i.e. arcs which ends at classical crossings and can contain virtual but not classical crossings.

\begin{definition}\label{def:virtual_index_cocycle}
{\em Virtual index cocycle} of the virtual diagram $D$ is the cochain $vi_D\in C^1(\tilde D,\Z)$, such that for any edge $e$ of $\tilde D$ the value $vi_D(e)$ is equal to the sum of signs of the virtual crossings in the corresponding long arc in $D$ according to the rule in Fig.~\ref{fig:virtual_crossing_sign}.
\end{definition}

\begin{figure}[h]
\centering\includegraphics[width=0.3\textwidth]{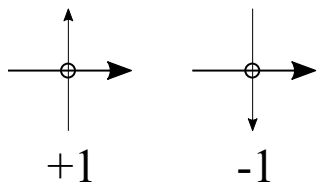}
\caption{The sign of a virtual crossing}\label{fig:virtual_crossing_sign}
\end{figure}

Virtual index cocycle $vi_D$ determines the cochain $(p_1)^*(vi_D)\in C^1(G(D),\Z)$ which we also denote as $vi_D$ for short. Note that $vi_D(c)=0$ for any chord $c$ in the Gauss diagram $G(D)$.

\begin{example}
The virtual index cocycle for a virtual trefoil diagram is shown in Fig.~\ref{fig:virtual_trefoil_virtual_index}.
\begin{figure}[h]
\centering\includegraphics[width=0.5\textwidth]{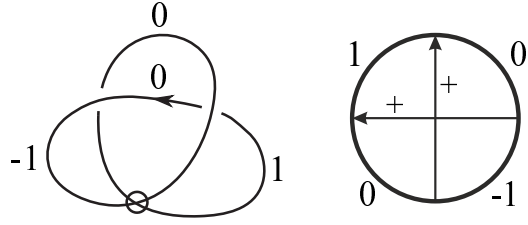}
\caption{The virtual index cocycle of a virtual trefoil diagram}\label{fig:virtual_trefoil_virtual_index}
\end{figure}
\end{example}

Let a virtual diagram $D'$ be obtained from $D$ by applying virtual Reidemeister moves. Then we can identify the graphs $\tilde D$ and $\tilde D'$ as well as the Gauss diagrams $G(D)$ and $G(D')$. Thus, we can compare the cochains $vi_D$ and $vi_{D'}$.

\begin{proposition}\label{prop:vi_property}
\begin{enumerate}
\item If the diagram $D'$ differs from $D$ by a move $VR_1,VR_2,VR_3$ then $vi_{D'}=vi_D$;
\item If the diagram $D'$ differs from $D$ by a move $SVR_3$ applied at a classical crossing $c$ then $vi_{D'}=vi_D \pm \delta c$ where $\delta$ is the differential in the cochain complex $C^*(\tilde D,\Z)$, see Fig.~\ref{fig:delta_crossing};
\item $vi_D(D)=0$.
\end{enumerate}
\end{proposition}

\begin{figure}[h]
\centering\includegraphics[width=0.12\textwidth]{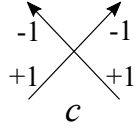}
\caption{The coboundary $\delta c$}\label{fig:delta_crossing}
\end{figure}

\begin{proof}
1.\ Let $e$ be an edge of $\tilde D$. If the corresponding long arc does not take part in the virtual Reidemeister move then $vi_D(e)=vi_{D'}(e)$.

If a move $VR_1$ is applied to $e$ then the new virtual crossing is accounted twice, with opposite signs. Hence $vi_D(e)=vi_{D'}(e)$.

If a move $VR_2$ is applied to $e$ then two new virtual crossings with opposite signs are added to $e$. Therefore $vi_{D'}(e)=vi_D(e)+1-1=vi_D(e)$.

The move $VR_3$ does not change the signs or the number of virtual crossings on long arcs.

2.\ Let us consider the move $SVR_3$. There are several cases of this move, see for example Fig.~\ref{fig:SVR_move}. In any case the virtual index cocycle changes by $\pm\delta c$ depending on the orientations of the arcs.

\begin{figure}[h]
\centering\includegraphics[width=0.3\textwidth]{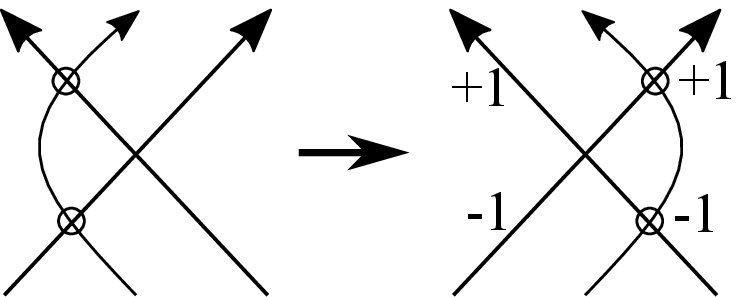}\qquad\qquad\includegraphics[width=0.3\textwidth]{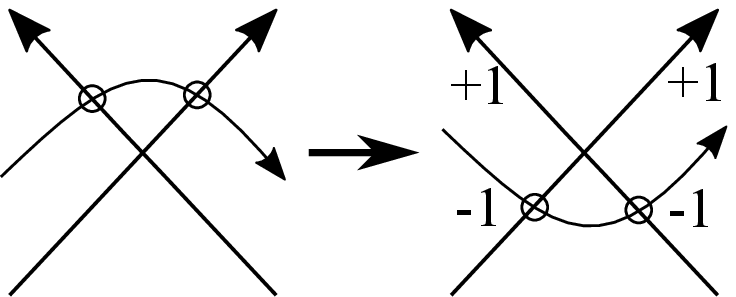}
\caption{Move $SVR_3$}\label{fig:SVR_move}
\end{figure}

3.\ In the sum $v_D(D)=\sum_{e\mbox{\scriptsize \ is a long arc of } D} vi_D(e)$ each virtual crossing contributes twice with different signs, so its total contribution to the sum is zero. Therefore, the sum is equal to $0$.
\end{proof}

\begin{corollary}
The cohomology class of the virtual index cocycle $vi_D$ in $H^1(\tilde D)$ (or in $H^1(G(D))$) does not change under virtual Reidemeister moves. Thus, the cohomology class is determined by the graph $\tilde D$ (or the Gauss diagram $G(D)$) but not by the placement of virtual crossings in the diagram $D$.
\end{corollary}

\begin{definition}\label{def:virtual_index_class}
We call the cohomology class $[vi_D]\in H^1(\tilde D,\Z)\cong H^1(G(D),\Z)$  the {\em virtual index class} of the diagram $D$, and denote it by
$\iota_D = \iota_{\tilde D} = \iota_{G(D)}$.
\end{definition}

The class $\iota_D$ has different representatives besides the cocycle $vi_D$. The following statement says that any representative is realizable.

\begin{proposition}
Let $\alpha\in C^1(\tilde D,\Z)$ be a representative of the virtual index class $\iota_{\tilde D}$. Then there exists a virtual diagram $D'$ with the same covering graph $\tilde D$, such that $vi_{D'}=\alpha$.
\end{proposition}

\begin{proof}
Since $[\alpha]=\iota_{\tilde D}$, we have $\alpha=vi_D+\sum_{c_i}\pm\delta(c_i)$ for some sequence of crossings $c_i$. We construct $D'$ by applying virtual Reidemeister moves to the diagram $D$. It is sufficient to consider the case $\alpha=vi_D\pm\delta(c)$. The construction of $D'$ is shown in Fig.~\ref{fig:index_representatives}.
\begin{figure}[h]
\centering\includegraphics[width=0.6\textwidth]{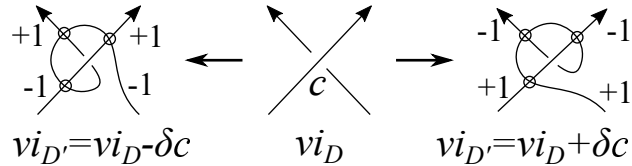}
\caption{Construction of the diagram $D'$}\label{fig:index_representatives}
\end{figure}
\end{proof}

\subsection{The canonical index cocycle}

Let us construct a representative of the virtual index class which is independent from the location of virtual crossings.

\begin{definition}
For any classical crossing $c$ of the diagram consider the cochains $\beta^l(c)$ and $\beta^r(c)$ in $C^1(\tilde D,\Z)$ as shown in Fig.~\ref{fig:left_right_cochains} (the labels of the edges which are not incident to $c$ are zero). Define the {\em left (right) canonical index cocycle} as the sum $ci^l_D=\sum_c \beta^l(c)$ ($ci^r_D=\sum_c \beta^r(c)$).

\begin{figure}[h]
\centering\includegraphics[width=0.25\textwidth]{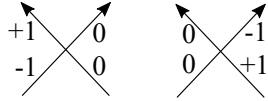}
\caption{The left and right cochains $\beta^l(c)$ and $\beta^r(c)$ (the choice of under-overcrossing does not matter)}\label{fig:left_right_cochains}
\end{figure}
\end{definition}

Note that the canonical cocycles are determined by the graph $\tilde D$ or the Gauss diagram $G(D)$ and don't rely on virtual crossings.

\begin{proposition}\label{prop:left_right_canonical_cocycles}
The left and the right canonical cocycles coincide: $$ci^l_D=ci^r_D.$$
\end{proposition}

\begin{proof}
Consider a long arc in the diagram. It ends in classical crossings. There are four cases of link orientations in these crossings (we disregard the over-undercrossing structure here). In all cases the left and the right canonical cocycles give the same value on the arc, see Fig.~\ref{fig:left_right_cochains_equality}.

\begin{figure}[h]
\centering\includegraphics[width=0.7\textwidth]{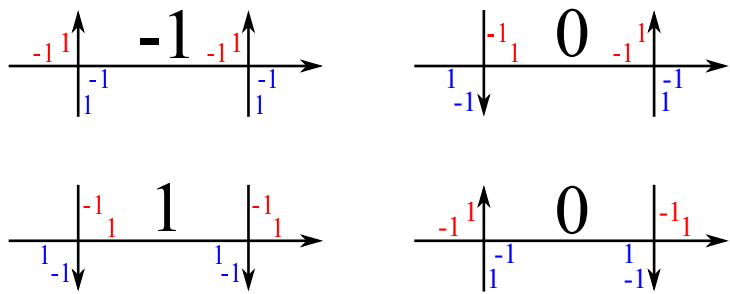}
\caption{Values of the canonical index cocycle on a long arc (the red marks are the left cocycle,  the blue ones are the right) }\label{fig:left_right_cochains_equality}
\end{figure}
\end{proof}

Below we denote the canonical cocycle by $ci_D=ci^l_D=ci^r_D$ or $ci_{\tilde D}$ or $ci_{G(D)}$.

It is time to formulate the central result of the paper. The theorem below states that the virtual index cocycle and virtual link invariants based on it can be reformulated without using information on virtual crossings.

\begin{theorem}\label{thm:canonical_index_cocycle}
Let $D$ be a diagram of an oriented virtual link. Then the canonical index cocycle $ci_D$ is homologous to the virtual index cocycle $vi_D$ in $C^1(\tilde D,\Z)$. In other words, the cocycle $ci_D$ is a representative of the index class $\iota_D$.
\end{theorem}

\begin{proof}
According to Proposition~\ref{prop:vi_property} it is suffice to construct a diagram $D'$ which can be obtained from $D$ by virtual Reidemeister moves, such that $vi_{D'}=ci_D$.

Given a virtual link diagram $D$, take an edge on each of its component and form an arc that contains only virtual crossings and goes along the component in the opposite direction on the left of the component, see Fig.~\ref{fig:virtual_double_diagram}. We denote the obtained diagram by $VD(D)$ and call it a {\em (left) virtual double} diagram of $D$.

\begin{figure}[h]
\centering\includegraphics[width=0.4\textwidth]{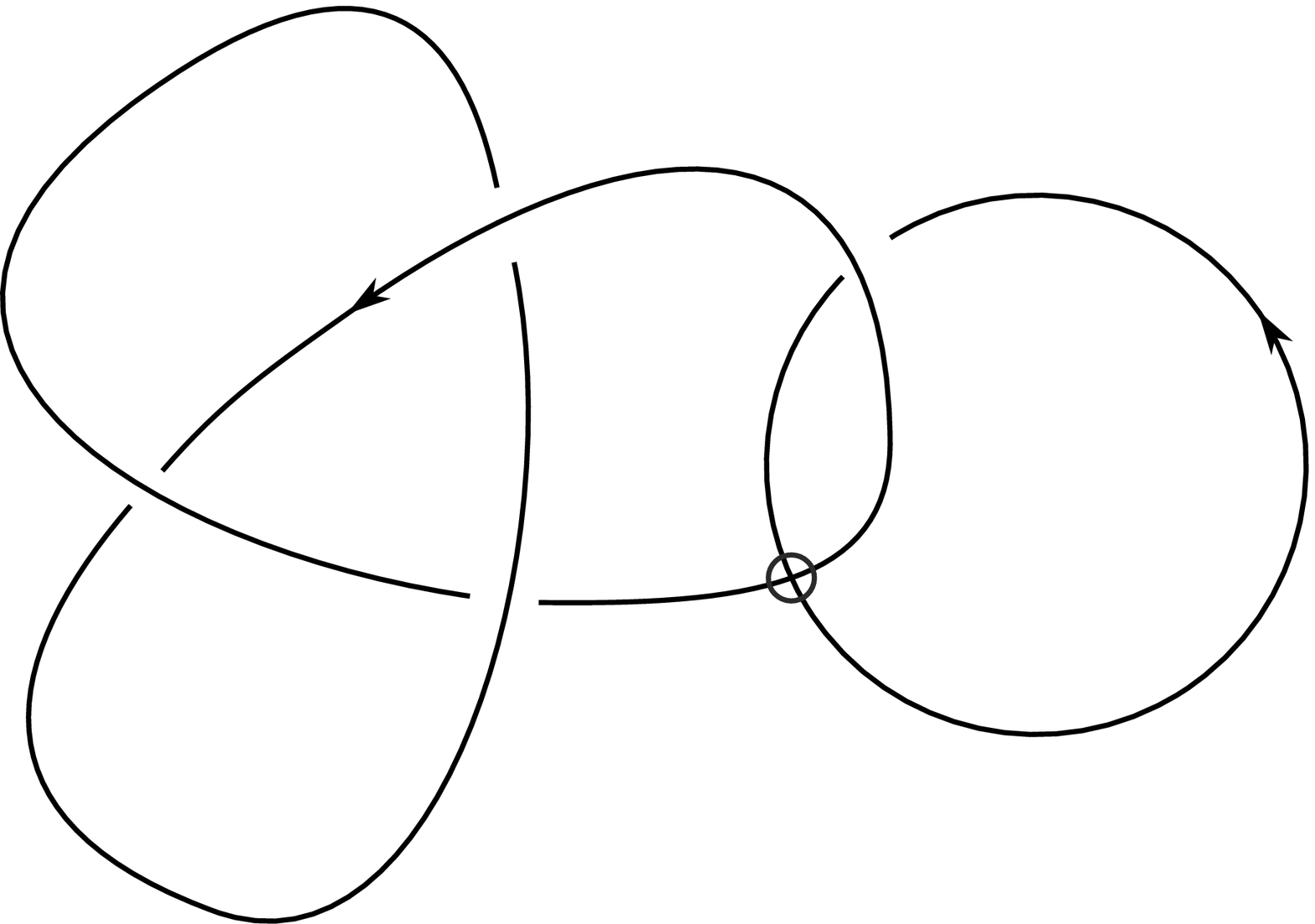}
\quad\centering\includegraphics[width=0.4\textwidth]{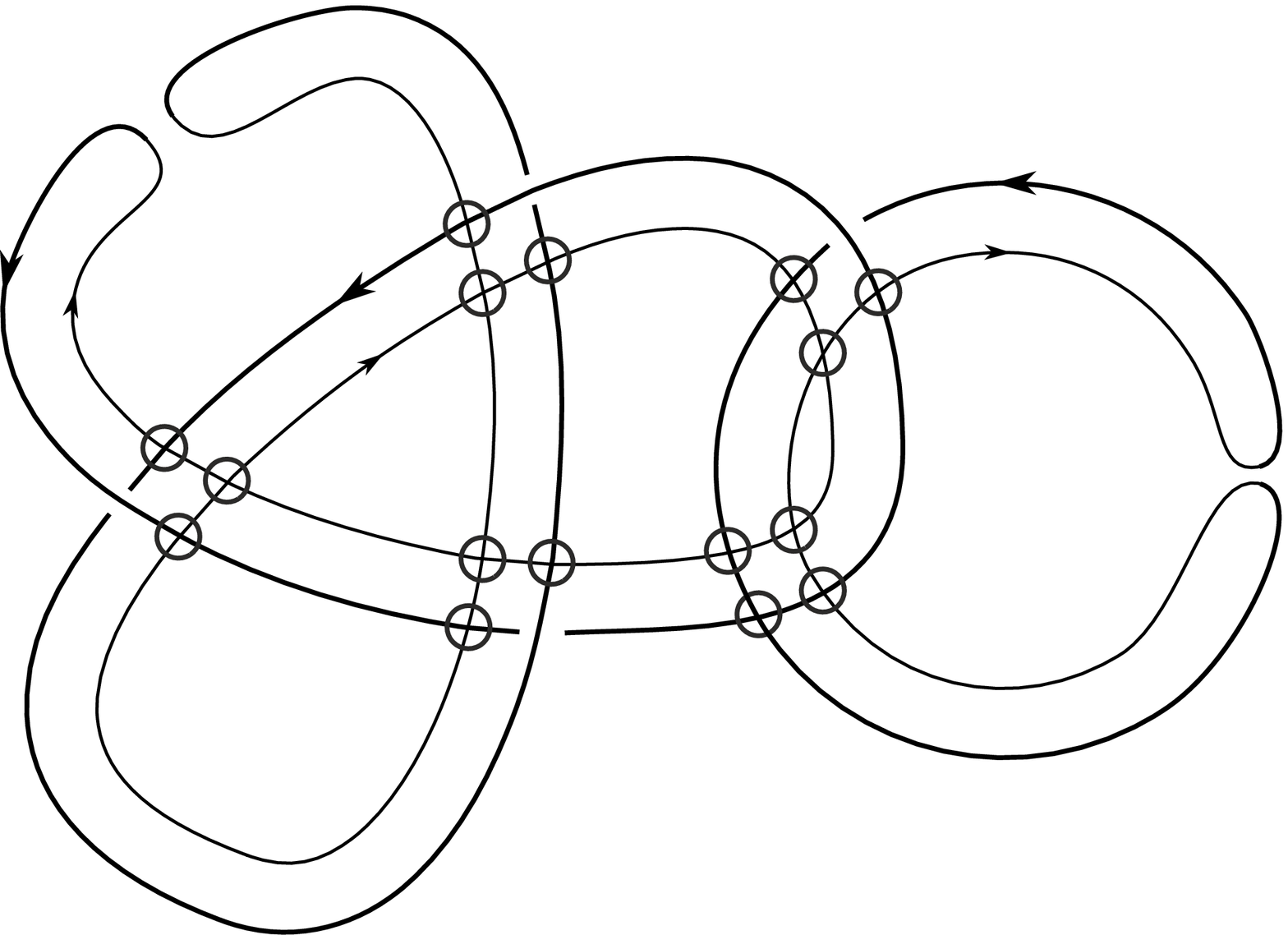}
\caption{A virtual link diagram and its left virtual double}\label{fig:virtual_double_diagram}
\end{figure}

Let us check that $vi_{VD(D)}=ci_D$. The contribution of every virtual crossing of $D$ vanishes in the virtual double because the neighbourhood of the crossing in the virtual double contains pairs of opposite virtual crossings, see Fig.~\ref{fig:classical_virtual_double_crossings} right. A classical crossing of $D$ produces the labels of the cochain $\beta^l(c)$, see Fig.~\ref{fig:classical_virtual_double_crossings} left and Fig.~\ref{fig:left_right_cochains} left. The virtual crossings on the pure virtual arc that parallel to the initial diagram $D$, occur in pairs and cancel each other.
\begin{figure}[h]
\centering\includegraphics[width=0.3\textwidth]{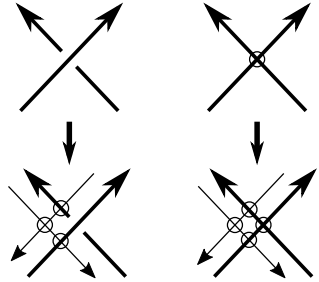}
\caption{A classical and a virtual crossings in the virtual double}\label{fig:classical_virtual_double_crossings}
\end{figure}

Thus, $vi_{VD(D)}=\sum_c \beta^l(c)=ci^l_D=ci_D$.
\end{proof}

\subsection{Virtual index cocycle and the (virtual) index}\label{sect:virtual_index}

Let $D$ be an oriented virtual {\em knot} diagram.

Let $c$ be a classical crossing of the diagram $D$. 
Then the diagram 
splits into the {\em left} and the {\em right halves} $D^l_c$ and $D^r_c$, see Fig.~\ref{fig:knot_halves}. Also, we introduce notation for the {\em signed halves} $D^\pm_c$ where $D^{\sign(c)}_c=D^l_c$, $D^{-\sign(c)}_c=D^r_c$.

One can assign to the crossing $c$ two numbers: the virtual~\cite{Im2013} and the classical~\cite{Cheng2013} indices.

\begin{figure}[h]
\centering\includegraphics[width=0.5\textwidth]{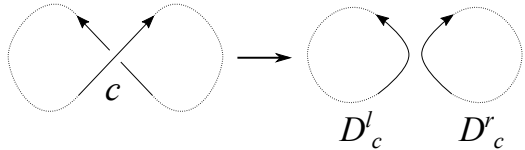}
\caption{Halves of the diagram}\label{fig:knot_halves}
\end{figure}

\begin{definition}\label{def:virtual_index}
The {\em virtual index} $vInd_D(c)$ of a crossing $c$ is the number $vi_D(D^+_c)$.
\end{definition}

\begin{remark}
The virtual index $vInd_D(c)$ is the same as the virtual intersection index $i(x)$ in~\cite{Im2013}.
\end{remark}

\begin{definition}[\cite{Cheng2013}]\label{def:index}
Let $D$ be an oriented virtual knot diagram and $c$ be a classical 
crossing of $D$. We count the sum of the intersection numbers (see Fig.~\ref{fig:classical_crossing_sign}) of the classical crossings in the positive half $D^+_c$. This sum $Ind_D(c)$ is called the {\em index} of the crossing $c$.

\begin{figure}[h]
\centering\includegraphics[width=0.6\textwidth]{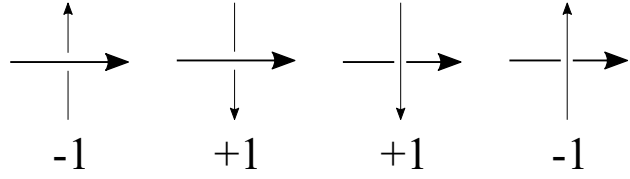}
\caption{Intersection number of a classical crossing}\label{fig:classical_crossing_sign}
\end{figure}
\end{definition}

\begin{proposition}\label{prop:index_and_canonical_cocycle}
For any classical 
crossing $c$ of the diagram $D$ the index $Ind_D(c)$ is equal to $ci_D(D^+_c)$.
\end{proposition}

\begin{proof}
Indeed, for any crossing $c'\ne c$ in $D$ the value $\beta^l_{c'}(D^+_c)$ is zero if $c'$ does not belong to the cycle $D^+_c$, and is equal to $\pm 1$ if $c'$ belong to $D^+_c$. Moreover, the sign coincides with that in Fig.~\ref{fig:classical_crossing_sign}. The value $\beta^l_{c}(D^+_c)$ is equal either $-1+1=0$ or $0+0=0$. Thus, $ci_D(D^+_c)=\sum_{c'}\beta^l_{c'}(D^+_c)$ is the index $Ind_D(c)$.
\end{proof}

\begin{corollary}\label{cor:index_virtual_and_classical}
The virtual index $vInd_D(c)$ and the index $Ind_D(c)$ coincide.
\end{corollary}

\begin{proof}
Theorem~\ref{thm:canonical_index_cocycle}, Definition~\ref{def:virtual_index} and Proposition~\ref{prop:index_and_canonical_cocycle} imply that $vInd_D(c)=vi_D(D^+_c)=ci_D(D^+_c)=Ind_D(c)$.
\end{proof}

Below we introduce another representative of the virtual index class which does not rely on virtual crossings. Consider the virtual index cocycle $vi_D$ as an element in $C^1(G(D),\Z)$. The Gauss diagram $G=G(D)$ as a graph consists of the oriented chords $c\in Ch(G)$ and the oriented core cycle $CC(G)$.

\begin{proposition}\label{prop:chord_index_cocycle}
Let $D$ be a virtual knot diagram and $G=G(D)$ be its Gauss diagram, and $\iota_D=\iota_{G}\in H^1(G,\Z)$ be the index class. There exists a unique cocycle $i^{ch}_{G}\in  C^1(G,\Z)$ such that $[i^{ch}_G]=\iota_G$ and $i^{ch}_G\mid_{CC(G)}\equiv 0$. Moreover, for any chord $c\in Ch(G)$ we have $i^{ch}_G(c)=Ind_D(c)$.
\end{proposition}
\begin{proof}
Let us show first that such cocycle $i^{ch}$ exists. Take any virtual index cocycle $vi_D\in C^1(G,\Z)$. Let $e$ be an edge of the core cycle with an end $c_+$, and $vi_D(e)=x$, see Fig.~\ref{fig:cocycle_reduction}. By adding the coboundary $x\delta c_+$, we can nullify the edge $e$ so that the value of the preceding edge changes. Repeating this operation, we get a cocycle $\alpha$ with all the edges of the core cycle nullified except may be one edge $e'$. But the nullification operation does not change the sum $\sum_{e\in CC(G)}vi_D(e)$ and the initial sum is zero by Proposition~\ref{prop:vi_property}. Thus, $0=\sum_{e\in CC(G)}\alpha(e)=\alpha(e')$, so $\alpha\mid_{CC(G)}\equiv 0$.

\begin{figure}[h]
\centering\includegraphics[width=0.5\textwidth]{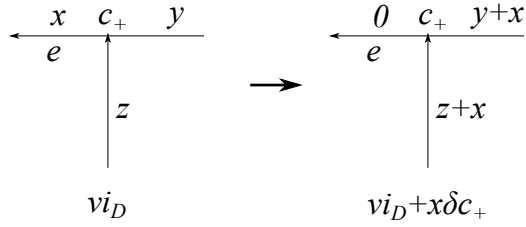}
\caption{Nullification of an edge in a virtual index cocycle}\label{fig:cocycle_reduction}
\end{figure}

Let $c\in Ch(G)$ be a chord. The ends of the chord splits the core cycle into two arcs $\gamma^+_c$ and $\gamma^-_c$, see Fig.~\ref{fig:gauss_halves}. Let $\hat\gamma^+_c$ be the cycle which is the union of the arc $\gamma^+_c$ and the chord $c$. The projection from $G$ to the diagram $D$ maps $\gamma^+_c$ and $\hat\gamma^+_c$ to the half $D^+_c$. Hence, $vi_D(\gamma^+_c)= vi_D(D^+_c)=Ind_D(c)$. Since $vi_D$ vanishes on any chord in $G$, $vi_D(\hat\gamma^+_c)=vi_D(\gamma^+_c)+vi_D(c)=vi_D(\gamma^+_c)=Ind(c)$. Adding a coboundary does not change the value of a cocycle on a cycle, so we have $\alpha(\hat\gamma^+_c)=Ind_D(c)$. But $\alpha(\hat\gamma^+_c)=\alpha(\gamma^+_c)+\alpha(c)=\alpha(c)$ because $\alpha$ is zero in the core cycle. Thus, $\alpha(c)=Ind_D(c)$. The values on the chords determine uniquely the cocycle $\alpha$.

\begin{figure}[h]
\centering\includegraphics[width=0.3\textwidth]{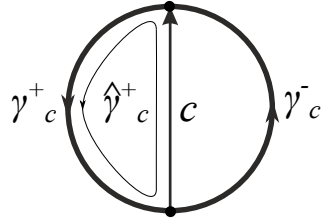}
\caption{The cycle $\hat\gamma^+_c$}\label{fig:gauss_halves}
\end{figure}
\end{proof}

\begin{definition}\label{def:chord_index_cocycle}
We call the cocycle $i^{ch}_G$ the {\em chord index cocycle}.
\end{definition}

\begin{example}
The virtual, canonical and chord index cocycles of a diagram of the virtual trefoil is shown in Fig.~\ref{fig:virtual_index_representatives}.
\begin{figure}[h]
\centering\includegraphics[width=0.5\textwidth]{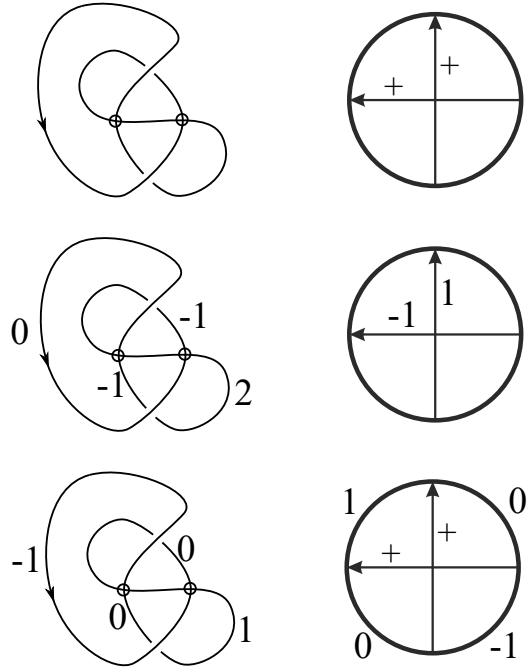}
\caption{A diagram of the virtual trefoil, its Gauss diagram (top row), the virtual index cocycle (middle left), the chord index cocycle (middle right), and the canonical index cocycle (bottom row)}\label{fig:virtual_index_representatives}
\end{figure}
\end{example}

\subsection{Wriggle number}

Let $D=D_1\cup\cdots\cup D_n$ be an oriented virtual link diagram with $n$ unicursal components $D_1,\dots,D_n$.

Let $c$ is be a {\em self-crossing} of $D$, i.e. a self-intersection of some unicursal component $D_i$ of $D$. Like in the knot case, the oriented smoothing in the crossing $c$ splits the component $D_i$ into two halves $D^l_c$ and $D^r_c$, see Fig.~\ref{fig:knot_halves}. Then we can define the index $Ind_D(c)$ and the virtual index $vInd_D(c)$ of the self-crossing $c$ as in Definitions~\ref{def:virtual_index} and~\ref{def:index}. Proposition~\ref{prop:index_and_canonical_cocycle} and Corollary~\ref{cor:index_virtual_and_classical} remain true for self-crossings of virtual links.

Proposition~\ref{prop:chord_index_cocycle} is not valid for link diagrams. The reason is in general case we don't have the equality $vi_D(D_i)=0$ for a unicursal component $D_i$ of the diagram $D$, so we can not nullify the cocycle on the corresponding core cycle of the Gauss diagram. This obstruction is regulated by the wriggle numbers of the link.

\begin{definition}[\cite{Folwaczny2013}]\label{def:wriggle_number}
Let $D=D_1\cup\cdots\cup D_n$ be an oriented virtual link diagram with $n$ unicursal components $D_1,\dots,D_n$. For any component $D_i$ define its {\em over linking number} $\lk^o_D(D_i)$ and {\em under linking number} $\lk^u_D(D_i)$ as the sums over all classical crossings with $D_i$:
\begin{equation}\label{eq:over_under_linking_number}
\lk^o_D(D_i)=\sum_{c\colon D_i\mbox{\scriptsize\ is overcrossing}}\sign(c),\qquad \lk^u_D(D_i)=\sum_{c\colon D_i\mbox{\scriptsize\ is undercrossing}}\sign(c).
\end{equation}
The {\em wriggle number} of the component $D_i$ is the difference
\begin{equation}\label{eq:wriggle_number}
w_D(D_i)=\lk^o_D(D_i)-\lk^u_D(D_i).
\end{equation}
\end{definition}

The over and under linking numbers are link invariants, so is the wriggle number~\cite{Folwaczny2013}.

\begin{proposition}\label{prop:wriggle_number_index_cocycle}
Let $\iota_D$ be the index class of the diagram $D$. Then $w_D(D_i)=-\iota_D(D_i)$.
\end{proposition}

\begin{proof}
We take the canonical cocycle as a representative of $\iota_D$. Then $\iota_D(D_i)=ci_D(D_i)$. For each crossing $c$ in $D_i$ its contribution to $ci_D(D_i)$ (see Fig.~\ref{fig:left_right_cochains}) is opposite to the contribution to $w_D(D_i)$. Thus, $ci_D(D_i)=-w_D(D_i)$.
\end{proof}

Let $D_i$ and $D_j$ be two components of $D$. Consider the wriggle number of the components $D_i$ and $D_j$: $w(D_i,D_j)=w_{D_i\cup D_j}(D_i)$. From Definition~\ref{def:wriggle_number} and Proposition~\ref{prop:wriggle_number_index_cocycle} it follows

\begin{proposition}
\begin{enumerate}
\item $w(D_i,D_j)=-\iota_{D_i\cup D_j}(D_i)$;
\item $w(D_i,D_j)$ is an invariant of links with numbered components;
\item $w(D_j,D_i)=-w(D_i,D_j)$;
\item $w_D(D_i)=\sum_{j=1}^n w(D_i,D_j)$.
\end{enumerate}
\end{proposition}

\begin{example}
Let $D=D_1\cup D_2$ is the virtual Hopf link, see Fig.~\ref{fig:virtual_hopf}. Then $w(D_1,D_2)=1$, $w(D_2,D_1)=-1$.

\begin{figure}[h]
\centering\includegraphics[width=0.25\textwidth]{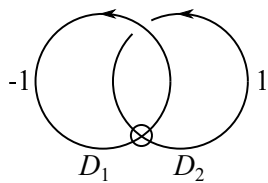}
\caption{Virtual Hopf link. The edges are labeled with the index cocycle}\label{fig:virtual_hopf}
\end{figure}
\end{example}


\subsection{Diagram smoothings}

Let $D$ be a virtual link diagram and $\X(D)$ be the set of classical crossings of $D$. For any crossing $c\in\X(D)$ define the {\em smoothings} of the diagram $D$ as shown in Fig.~\ref{fig:crossing_smoothings}.

\begin{figure}[h]
\centering\includegraphics[width=0.4\textwidth]{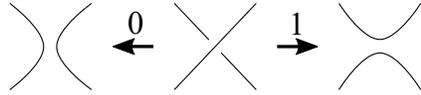}
\caption{Smoothings of a crossing}\label{fig:crossing_smoothings}
\end{figure}

Any map $s\colon\X(d)\to\{0,1\}$ determines smoothings in each crossing of $D$. The result of the smoothings is a diagram $D_s$ without classical crossings, see Fig.~\ref{fig:virtual_trefoil_with_state}. The diagram $D_s$ is called a {\em Kauffman state} of the diagram $D$.

\begin{figure}[h]
\centering\includegraphics[width=0.5\textwidth]{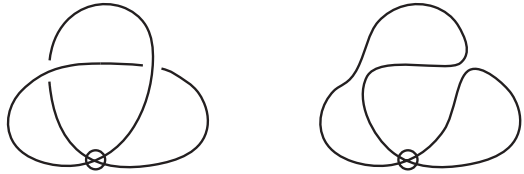}
\caption{Virtual trefoil and its Kauffman state}\label{fig:virtual_trefoil_with_state}
\end{figure}

Let $\Gamma(s)$ is the set of unicursal components of the state $D_s$. Any component $\gamma\in\Gamma(s)$ can be considered as a cycle in the covering graph $\tilde D$, so we apply the index cocycle $vi_D$ to it. Since $\gamma$ is a cycle, for any representative $\alpha\in\iota_D$ of the index class the value $\alpha(\gamma)$ is equal $vi_D(\gamma)$.

\begin{proposition}\label{prop:smoothing_cocycle_property}
Let $D_s$ be a Kauffman state of a virtual link diagram, and $\gamma\in\Gamma(s)$ be a component of it. Then $vi_D(\gamma)=0$.
\end{proposition}

\begin{proof}
The smoothed diagram $D_s$ is a plane diagram with only virtual crossing. Since the cochain $vi_D$ is defined by the rule in Fig.~\ref{fig:virtual_crossing_sign}, the value $vi_D(\gamma)$ coincides with the intersection number $\gamma\cdot D_s$. But the intersection number of any two cycles in $\R^2$ is zero.
\end{proof}

Proposition~\ref{prop:smoothing_cocycle_property} can be reformulated as a cocycle condition on some closed surface. Recall that the {\em atom}~\cite{Manturov2000} or {\em abstract link diagram}~\cite{KK2000} can be constructed for a virtual link diagram $D$ as follows. Consider a two-dimensional surface with boundary immersed in the neighbourhood of the diagram $D$ as shown in Fig.~\ref{fig:atom}. Then glue discs to the boundary of this surface. The obtained closed surface $A(D)$ is the atom of the diagram $D$.

\begin{figure}[h]
\centering\includegraphics[width=0.4\textwidth]{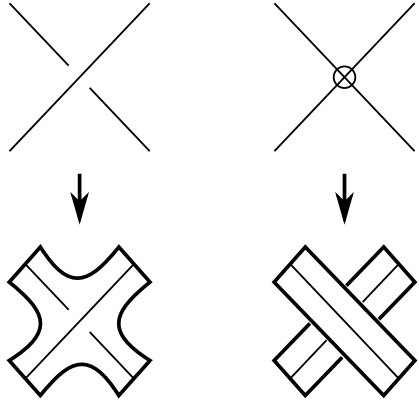}
\caption{Construction of the atom}\label{fig:atom}
\end{figure}

The covering graph $\tilde D$ embeds naturally in the surface $A(D)$ as its $1$-skeleton, so we can consider the index cocycle $vi_D$ as a $1$-cochain on $A(D)$.

\begin{corollary}\label{prop:atom_cocycle_property}
The virtual index cochain $vi_D\in C^1(A(D),\Z)$ is a genuine cocycle, i.e. for any $2$-cell $e\in A(D)$ we have $vi_D(\partial e)=0$.
\end{corollary}

\begin{proof}
Indeed, the boundary of any $2$-cell is a component of some smoothing $D_s$.
\end{proof}

\section{Invariants of virtual knots}\label{sect:virtual_invariants}

Below we mention two virtual link invariants which involve virtual crossings in their definition, and we show how these invariants can be reformulated without virtual crossings. The first example is the virtual intersection index polynomial~\cite{ILL} that is eqivalent to the index polynomial~\cite{Henrich2010,Cheng2013}. The other example is the virtual Alexander quandle~\cite{Manturov2002, Manturov2003,KauffmanManturov2005}.  The corresponding Alexander polynomial coincides with the polynomial induced by the Alexander biquandle~\cite{KauffmanManturov2005,Saw,SW}.

\subsection{Virtual index polynomial}

Let $D=D_1\cup\cdots\cup D_n$ be an oriented virtual link diagram with $n$ unicursal components $D_1,\dots,D_n$. Let $S(D)$ be the set of classical self-crossings of $D$, and $M(D)$ be the set of classical mixed crossings, i.e. intersections of different components of $D$. Following~\cite{IKP}, define the {\em virtual intersection index} of a self-crossing $c$ of a component $D_i$  as $vi_{D_i}(D^+_c)=vInd_{D_i}(c)$. If $c$ is a mixed crossing, i.e. an intersection of components $D_i$ and $D_j$, and $D_i$ is overcrossing in $c$, then the {\em virtual intersection index} is defined as $i(c)=vi_{D_i\cup D_j}(D_i)$. Then $i(c)=-w(D_i,D_j)$.

Consider the polynomial invariant $Q_D(x,y)$ defined in~\cite{IKP}:
$$
Q_D(x,y)=\sum_{c\in S(D)}\sign(c)(x^{i(c)}-1)+\sum_{c\in M(D)}\sign(c)(y^{i(c)}-1).
$$
We can reformulate it using the index cocycle as follows.

\begin{multline*}
Q_D(x,y)=\sum_{i=1}^n\sum_{c\in S(D_i)}\sign(c)x^{\iota_{D_i}(D^+_c)}+\sum_{i=1}^n\sum_{j=1}^n \lk^o_{D_i\cup D_j}(D_i)y^{\iota_{D_i\cup D_j}(D_i)}\\-writhe(D)=\sum_{i=1}^n\sum_{c\in S(D_i)}\sign(c)x^{Ind_{D_i}(c)}+\sum_{i=1}^n\sum_{j=1}^n \lk^o_{D_i\cup D_j}(D_i)y^{-w(D_i, D_j)}\\-writhe(D)
\end{multline*}
where $writhe(D)$ is the sum of signs of all the classical crossings of $D$.

Thus, we have got a known result~\cite{ILL} that the virtual intersection index polynomial coincide with the index polynomial.

\subsection{Virtual biquandle}

\begin{definition}\label{def:biquandle}
A set $B$ with two binary operations $\circ,\ast\colon B\times B\to B$ is called a {\em biquandle}~\cite{Nelson2015} if it obeys the following conditions:
\begin{enumerate}
\item $x\circ x=x\ast x$ for any $x\in B$;
\item for any $y\in B$ the operators $\alpha_y\colon B\to B, x\mapsto x\circ y$, and $\beta_y\colon B\to B, x\mapsto x\ast y$, are invertible;
\item the map $S\colon B\times B\to B\times B$, $(x,y)\mapsto (x\circ y, y\ast x)$, is a bijection;
\item for any $x,y,z\in B$
\begin{gather*}
(x\circ z)\circ(y\circ z)=(x\circ y)\circ(z\ast y),\\
(x\circ z)\ast(y\circ z)=(x\ast y)\circ(z\ast y),\\
(x\ast z)\ast(y\ast z)=(x\ast y)\ast(z\circ y).
\end{gather*}
\end{enumerate}

 A {\em virtual biquandle}~\cite{KauffmanManturov2005} is a pair $(B,f)$ where $B$ is a biquandle and $f\colon B\to B$ is an automorphism of $B$, i.e. $f(x\circ y)=f(x)\circ f(y)$ and  $f(x\ast y)=f(x)\ast f(y)$ for any $x,y\in B$.
\end{definition}

\begin{definition}
Let $D$ be a diagram of an oriented virtual link and $B$ be a biquandle. A {\em colouring of the diagram $D$ with the biquandle $B$} is a map from the set of long arcs of $D$ to $B$ such that the images of the arcs (colours) satisfy at each classical crossing the relations in Fig.~\ref{fig:biquandle_coloring}. Denote the set of colourings by $Col_B(D)$.
\begin{figure}[h]
\centering\includegraphics[width=0.4\textwidth]{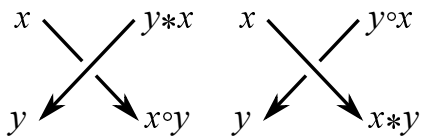}
\caption{The colouring rule}\label{fig:biquandle_coloring}
\end{figure}

Let $(B,f)$ be a virtual biquandle. Then a {\em coloring of the diagram $D$ with the virtual biquandle $(B,f)$} is a map from the set of short arcs of $D$ (whose ends are classical or virtual crossings) to $B$ such that the images of the arcs (colours) satisfy the relations in Fig.~\ref{fig:biquandle_coloring} at each classical crossing and the relation in Fig.~\ref{fig:virtual_biquandle_coloring} at each virtual crossing. Let $Col_{(B,f)}(D)$ be the set of colourings.

\begin{figure}[h]
\centering\includegraphics[width=0.2\textwidth]{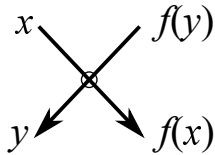}
\caption{The virtual colouring rule}\label{fig:virtual_biquandle_coloring}
\end{figure}
\end{definition}

\begin{remark}\label{rem:biquandle_reidemeister}
Given a biquandle, Reidemeister moves on virtual diagrams induce bijections of the sets of colourings of the diagrams with the biquandle. Thus, the number of colourings is a virtual link invariant. The same is valid for virtual biquandles.
\end{remark}

\begin{example}[Alexander biquandle]
Let $D$ be a diagram of an oriented virtual link. Let $ABQ(D)$ be a module over the ring $A=\Z[t,t^{-1},s,s^{-1}]$ whose generators are the long arcs of the diagram $D$, and the relations appear at the classical crossings of $D$ as shown in Fig.~\ref{fig:alexander_biquandle}. Then $ABQ(D)$ is a biquandle with operations $x\circ y=s^{-1}tx+s^{-1}(1-t)y$ and $x\ast y =s^{-1}x$, and its definition determines a tautological colouring of the diagram $D$ which maps any long arc to the corresponding generator of $ABQ(D)$.  The Fitting ideal of the module $ABQ(D)$ is generated by a polynomial $G_D(s,t)$ which is called the {\em generalized Alexander polynomial} of $D$. The polynomial $G_D(s,t)$ is an invariant of virtual links which vanishes on classical links~\cite{KauffmanManturov2005,Saw,SW}.

\begin{figure}[h]
\centering\includegraphics[width=0.5\textwidth]{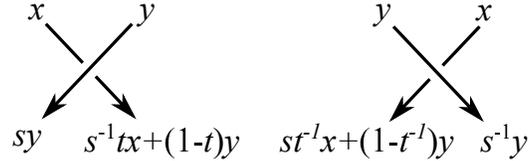}
\caption{Alexander biquandle relation}\label{fig:alexander_biquandle}
\end{figure}
\end{example}

\begin{remark}
Our notation differs from the one in~\cite[Section 2.3]{KauffmanManturov2005} by the variable change $t\mapsto st$.
\end{remark}

\begin{example}[Alexander virtual quandle]
Let $D$ be a diagram of an oriented virtual link and $VAQ(D)$ be the module over the ring $A=\Z[t,t^{-1},s,s^{-1}]$ whose generators are the short arcs of the diagram $D$, and the relations correspond to classical and virtual crossings of $D$ as shown in Fig.~\ref{fig:alexander_quandle}. The module $VAQ(D)$ has a virtual biquandle structure with operations $x\circ y=tx+(1-t)y$, $x\ast y=x$ and $f(x)=sx$. It also has a tautological colouring of the diagram $D$. The generator $\xi(D)$ of the Fitting ideal of the module $VAQ(D)$ is a polynomial invariant of virtual links~\cite{Manturov2003}.

\begin{figure}[h]
\centering\includegraphics[width=0.5\textwidth]{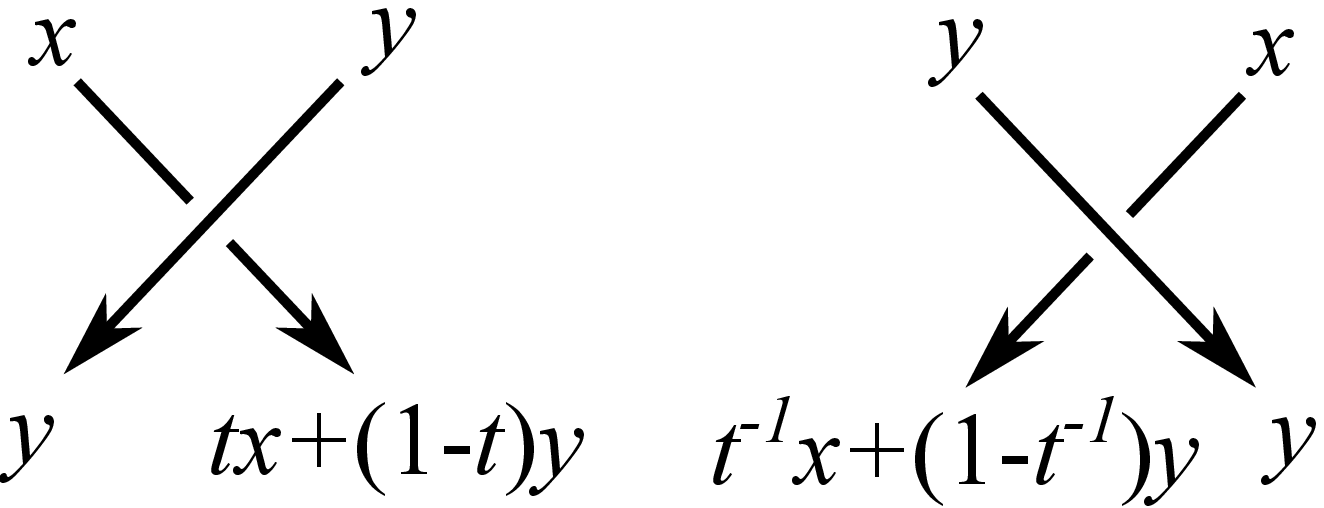}\qquad\qquad \includegraphics[width=0.18\textwidth]{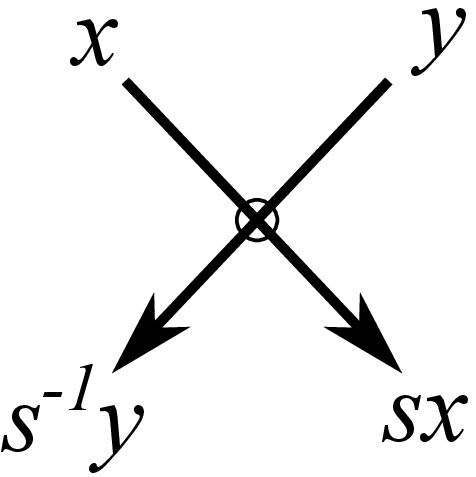}
\caption{Virtual Alexander quandle relation}\label{fig:alexander_quandle}
\end{figure}
\end{example}

Let $(B,f)$ be a virtual biquandle. Consider two binary operations $\circ_f,\ast_f$ given by the formula
\begin{equation}\label{eq:twisted_biquandle}
x\circ_f y = f^{-1}(x)\circ f^{-1}(y), \quad x\ast_f y = f^{-1}(x)\ast f^{-1}(y).
\end{equation}

\begin{theorem}\label{thm:twisted_biquandle}
1. The set $B$ with the operations $\circ_f$ and $\ast_f$ is a biquandle (we call it the {\em twisted biquandle} and denote $B_f$).

2. For any virtual diagram $D$ there is a bijection between the colouring sets $Col_{(B,f)}(D)$  and $Col_{B_f}(D)$.
\end{theorem}

\begin{proof}
1. The biquandle properties for $B_f$ are checked by a direct verification.

2. Let $D$ be a diagram of an oriented virtual link. We can assume that the virtual index cocycle of $D$ is canonical: $vi_D=ci_D$. Otherwise apply virtual Reidemeister moves to get such a diagram; by Remark~\ref{rem:biquandle_reidemeister} it has isomorphic sets of colourings. At each classical crossing we apply two first Reidemeister moves and choose four short arcs as shown in Fig.~\ref{fig:diagram_virtual_biquandle}. Then we obtain a diagram $D'$ such that $vi_{D'}=vi_D=ci_D$, and on any long arc of $D'$ there are two chosen short arcs (they coincide when the long arc has no virtual crossings).

\begin{figure}[h]
\centering\includegraphics[width=0.4\textwidth]{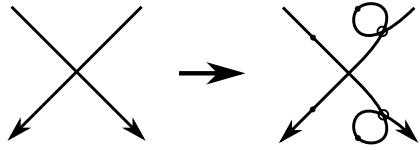}
\caption{Chosen short arcs in the diagram $D'$}\label{fig:diagram_virtual_biquandle}
\end{figure}

Consider a colouring of the diagram $D'$ with the virtual biquandle $(B,f)$. The condition $vi_{D'}=ci_D=ci^l_D$ ensures that the colours of the two chosen short arcs of any long arc of $D'$ coincide. At any crossing the colours of the chosen short arcs obey the relations of the biquandle $B_f$, see Fig.~\ref{fig:skewed_biquandle_operations}.  Hence, any colouring of $D'$ with the virtual biquandle $(B,f)$ induces a colouring of $D'$ with $B_f$: the colour of a long arc is the colour of the chosen short arcs in it.

\begin{figure}[h]
\centering\includegraphics[width=0.4\textwidth]{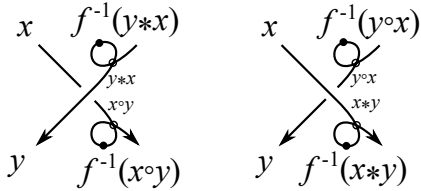}
\caption{Colouring of short arcs at a crossing of the diagram $D'$}\label{fig:skewed_biquandle_operations}
\end{figure}

Conversely, given a coloring of the diagram $D'$ with the twisted biquandle $B_f$, we colour the chosen short arcs with the colour of the long arc they belong to, and propagate the colouring to the other short arc using the operator $f$ (see Fig.~\ref{fig:virtual_biquandle_coloring}).

Thus, we have a bijection between the colouring sets $Col_{(B,f)}(D')$ and $Col_{B_f}(D')$. Hence,
$$Col_{(B,f)}(D)\cong Col_{(B,f)}(D')\cong Col_{B_f}(D')\cong Col_{B_f}(D).$$
\end{proof}

\begin{remark}
In~\cite[Section 2.6]{KauffmanManturov2005} L.H. Kauffman and V.O. Manturov introduced a more general construction of {\em formal virtual biquandle}. It would be interesting to find out whether formal virtual biquandles can be reduced to biquandles by means of some twist.
\end{remark}

As a consequence of Theorem~\ref{thm:twisted_biquandle} for the Alexander biquandle we have the following proposition which reproduces the conclusion of~\cite[Theorem 7.1]{BF2008}.

\begin{proposition}
Let $D$ be a diagram of an oriented virtual link. The Alexander biquandle $ABQ(D)$ and the virtual Alexander quandle $VAQ(D)$ are isomorphic as $\Z[t,t^{-1},s,s^{-1}]$-modules. The polynomials $G_D(s,t)$ and $\xi(D)$ coincide (up to multiplication by $\pm t^ks^l$).
\end{proposition}

\begin{proof}
Indeed, the formulas for the twisted biquandle structure on $VAQ(D)$ coincides with the Alexander biquandle in $ABQ(D)$. Using the tautological colorings we can identify the generators of the modules $VAQ(D)$ and $ABQ(D)$ (more precise, we identify long arc generators of $ABQ(D)$ with the chosen short arc generators of $VAQ(D)$ like in Theorem~\ref{thm:twisted_biquandle}). Since the relations in the modules are determined with the biquandle operators, the identification of the generators induces a well-defined isomorphism of the modules $VAQ(D)$ and $ABQ(D)$.

The polynomials $G_D(s,t)$ and $\xi(D)$ coincide (up to multiplication by an invertible element of the ring $\Z[t,t^{-1},s,s^{-1}]$) as the generators of the Fitting ideals of the Alexander modules.
\end{proof}

\section{Parity cocycle and local source-sink structures}

\subsection{Parity cocycle}

Let $D$ be an oriented virtual link diagram. Below we study the parity cocycle of $D$ which is just reduction of the virtual index cocycle mod $2$.

\begin{definition}\label{df:parity_cocycle}
The {\em virtual parity cocycle} $vp_D$ is the image of the virtual index cocycle $vi_D$ under the natural projection $C^1(\tilde D,\Z)\to C^1(\tilde D,\Z_2)$.
\end{definition}

\begin{remark}
1) In the same manner we can define {\em canonical parity cocycle} $cp_D=ci_D\pmod 2$, and the {\em parity class} $\pi_D=\iota_D\pmod 2\in H^1(\tilde D,\Z_2)$.

2) if $D$ is a knot diagram then for any classical crossing $vp_D(D^+_c)=vp_D(D^-_c)$ is the {\em Gaussian parity}~\cite{IMN} of the crossing $c$. This fact justifies the using word ``parity'' for $vp_D$.

3) the virtual parity cocycle $vp_D$ as well as the parity class $\pi_D$ can be defined for unoriented virtual link diagram. As shows Fig.~\ref{fig:left_right_cochains_equality}, the canonical parity cocycle $cp_D$ depends on the orientation, although it is well-defined  for unoriented virtual knots.
\end{remark}

\begin{remark}\label{rem:cocycle_cut_loci}
Since the value of $vp_D(e)$ for any long arc $e$ is equal to $0$ or $1$, we can picture the cocycle using {\em cut loci}: we place a cut locus on a long arc if its parity value is $1$ and place nothing if the value is $0$.

For example, the parity cocycles for a diagram of the virtual trefoil is shown in Fig.~\ref{fig:parity_representatives} (cf. Fig.~\ref{fig:virtual_index_representatives}).
\begin{figure}[h]
\centering\includegraphics[width=0.5\textwidth]{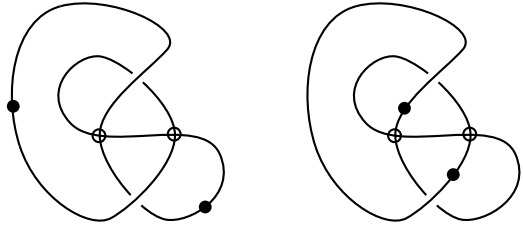}
\caption{The canonical parity cocycle (left) and the virtual parity cocycle (right) of a virtual trefoil diagram}\label{fig:parity_representatives}
\end{figure}

Below we shall consider diagrams with more than one cut locus on a long arc. In this case we treat them mod $2$, i.e. we allow two cut loci in one arc contract each other, see Fig.~\ref{fig:cut_loci_contraction}.

\begin{figure}[h]
\centering\includegraphics[width=0.4\textwidth]{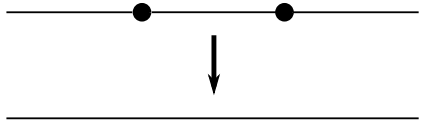}
\caption{Cut loci contraction}\label{fig:cut_loci_contraction}
\end{figure}
\end{remark}

Now we can reformulate Proposition~\ref{prop:smoothing_cocycle_property} for the virtual parity cocycle as follows.
\begin{corollary}\label{cor:parity_cocycle_property}
For any component $\gamma\in \Gamma(s)$ of any Kauffman state $D_s$ the number of cut loci in $\gamma$ is even.
\end{corollary}

\subsection{Source-sink structures}

Let $D$ be a (unoriented) virtual link diagram and $\X(D)$ be the set of classical crossings of $D$. For any crossing $c\in\X(D)$ we choose one of the two {\em source-sink orientations}~\cite{IM13} (see Fig.~\ref{fig:source_sink_orientations}). Denote the choice of source-link orientation at each classical crossing by $\lambda$ and call it a {\em local source-sink structure (LSSS)} of the diagram $D$. Denote the set of all local source-sink structures by $\Lambda(D)$.

\begin{figure}[h]
\centering\includegraphics[width=0.6\textwidth]{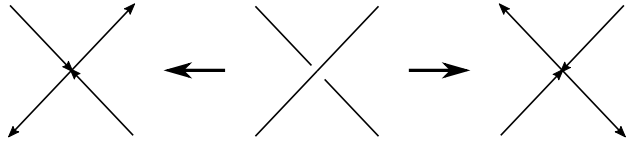}
\caption{Source-sink orientations}\label{fig:source_sink_orientations}
\end{figure}

\begin{example}
If the diagram $D$ is oriented, one can define the {\em canonical source-sink structure}~\cite{DKK2017} as shown in Fig.~\ref{fig:canonical_source_sink_orientation}.
\begin{figure}[h]
\centering\includegraphics[width=0.4\textwidth]{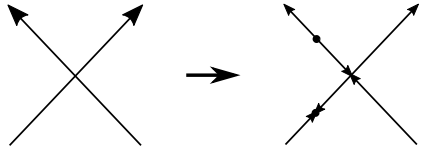}
\caption{Canonical source-sink orientation. The cut loci mark the places where the local source-sink orientation changes to the orientation of the link}\label{fig:canonical_source_sink_orientation}
\end{figure}

\end{example}

\begin{remark}
Source-sink structure can be used to define (local) orientation of Kauffman states of the diagram, see Fig.~\ref{fig:source_sink_smoothings}.
\begin{figure}[h]
\centering\includegraphics[width=0.6\textwidth]{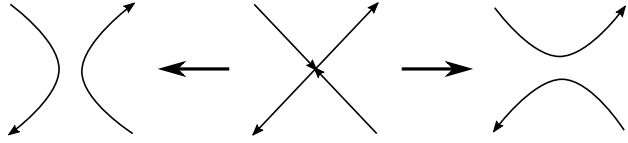}
\caption{Local orientation of smoothed components}\label{fig:source_sink_smoothings}
\end{figure}
\end{remark}

Any LSSS $\lambda$ has the {\em opposite LSSS} $-\lambda$ that is obtained by the {\em global change of orientation}, i.e. when one switches the source-sink structure at every classical crossing of $D$.

Given a LSSS $\lambda$, it defines a $1$-cochain $ssc_\lambda\in C^1(\tilde D,\Z_2)$ as follows. For any long arc $e$, if the orientations induced by $\lambda$ at the ends of $e$ coincide we set $ssc_\lambda(e)=0$, if the orientations are different we set $ssc_\lambda(e)=1$. We picture the cochain with cut loci like in Remark~\ref{rem:cocycle_cut_loci}, see Fig.~\ref{fig:virtual_trefoil_lsss}.
\begin{figure}[h]
\centering\includegraphics[width=0.3\textwidth]{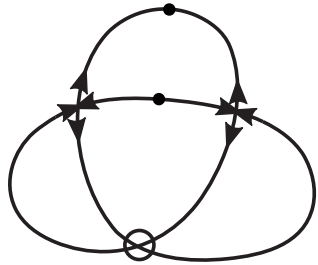}
\caption{A local source-sink structure on the virtual trefoil and the corresponding $1$-cochain}\label{fig:virtual_trefoil_lsss}
\end{figure}

If we take the opposite LSSS $-\lambda$, the orientations compatibility on any long arc is the same as for $\lambda$. Thus, $ssc_{-\lambda}=ssc_\lambda$.

Let the covering graph $\tilde D$ be connected. Then the following statement holds.

\begin{proposition}\label{prop:lsss_parity_cochains}
The map $ssc\colon\Lambda(D)\to C^1(\tilde D,\Z_2)$ induces a bijection between the set of local source-sink structures $\lambda$ considered up to global change of orientation and the set of representatives $\alpha$ of the parity class $\pi_D$.
\end{proposition}

\begin{proof}
We can assume that $D$ is oriented. Then we can compare the canonical LSSS $\lambda_{can}$ (see Fig.~\ref{fig:canonical_source_sink_orientation}) and the (left) canonical parity cocycle $cp_D$ (see Fig.~\ref{fig:left_right_cochains} left). We see that the places where the local orientation of $\lambda_{can}$ violates the global orientation of $D$ correspond exactly to the non-zero labels of the left canonical parity cocycle. Thus, $$ssc_{\lambda_{can}}=\sum_{c\in\X(D)}\beta^l(c)=cp_D\in C^1(\tilde D,\Z_2).$$

Any LSSS $\lambda$ differs from the canonical LSSS by changes of local source-sink orientation in some crossings, see Fig.~\ref{fig:source_sink_switch}. Let $C\subset\X(D)$ be the set of crossings where the source-sink orientation is changed. Each change of orientation adds cut loci to the arcs incident to the crossing (cut loci may be contracted in pairs after that). In other words, the change of orientation at a crossing $c$ adds to the cocycle $ssc_{\lambda_{can}}$ the cochain $\delta c\pmod 2$, see Fig.~\ref{fig:delta_crossing}. Hence,
$$
ssc_\lambda = ssc_{\lambda_{can}}+\sum_{c\in C}\delta c = cp_C+\sum_{c\in C}\delta c\in\pi_D.
$$

\begin{figure}[h]
\centering\includegraphics[width=0.4\textwidth]{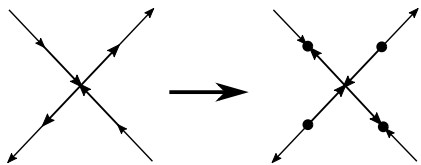}
\caption{Change of local source-sink orientation}\label{fig:source_sink_switch}
\end{figure}

Conversly, any representative of $\pi_D$ is of form $\alpha= cp_D+\sum_{c\in C}\delta c$ for some set of crossings $C\subset\X(D)$. Then $\alpha=ssc_\lambda$ where $\lambda$ is obtained from $\lambda_{can}$ by changing of local orientation in the crossings from the set $C$.

Finally, we need to show that if $ssc_{\lambda_1}=ssc_{\lambda_2}$ then either $\lambda_1=\lambda_2$ or $\lambda_1=-\lambda_2$. The cocycle $ssc_{\lambda_1}$ determines the local orientation of an end of a long arc if we know the orientation on the other end of the arc. Choose an arbitrary crossing and set an arbitrary local source-sink orientation in it. Since the diagram is connected, we can uniquely propagate the local orientation to all the other crossings. The ambiguity of the choice of the initial local orientation leads to two LSSS $\lambda_1$ and $-\lambda_1$. Since $ssc_{\lambda_1}=ssc_{\lambda_2}$, we have either $\lambda_1=\lambda_2$ or $\lambda_1=-\lambda_2$.
\end{proof}

\begin{corollary}
Let $D$ be a virtual link diagram. The atom $A(D)$ of the diagram $D$ is checkerboard colourable (i.e. one can colour the cells of $A(D)$ in two colours such that any edge of $\tilde D\subset A(D)$ is adjacent to two cells of different colour) if and only if $\pi_D=0$.
\end{corollary}

\begin{proof}
V.O. Manturov proved~\cite{IM13} that the atom is checkerboard colourable if and only if the diagram $D$ admits a global source-sink structure, i.e. there is a LSSS $\lambda$ without cut loci, i.e. $ssc_\lambda=0$. Hence, $\pi_D=[ssc_\lambda]=0$.

Conversely, if $\pi_D=0$ then the zero cocycle $\alpha=0$ is a representative of $\pi_D$. By Proposition~\ref{prop:lsss_parity_cochains}, there exists a LSSS $\lambda$ such that $ssc_\lambda=0$. Hence, $\lambda$ defines a global source-sink structure on the diagram $D$.
\end{proof}

\subsection{Khovanov homology of virtual links}

As an application of the parity cocycle we consider a construction of the Khovanov homology of virtual links which uses a local source-sink structure. The construction below is a reformulation of the Khovanov homology constructions in~\cite{DKK2017,Manturov2007,Manturov2007a}.

Let us recall that the Frobenius system used for calculation of Khovanov homology is the algebra $V=\Z[X]/(X^2)$ with the comultiplication
$$
\Delta(1)=1\otimes X+X\otimes 1,\qquad \Delta(X)=X\otimes X.
$$


Let $D$ be a diagram of a virtual link and $\X(D)$ be the set of the classical crossings of $D$. Denote the set of Kauffman states of the diagram $D$ by $S(D)=\{s\colon\X(D)\to\{0,1\}\}$.

In order to define Khovanov complex, we shall use the following additional structures:
\begin{itemize}
\item a LSSS $\lambda$ on the diagram $D$;
\item oriented directions $o$ on $\lambda$;
\item orders $\sigma_s$ on the sets of components $\Gamma(s)$ of Kauffman states;
\item star points $\star_s(\gamma)$ in all components $\gamma\in\Gamma(s)$ of Kauffman states.
\end{itemize}

Let us define the chains of the Khovanov complex.

Fix a LSSS $\lambda$ on the diagram $D$. It defines a $1$-cochain $ssc_\lambda$ which corresponds to a set of cut loci on the long arcs of the diagram $D$, see Fig.~\ref{fig:virtual_trefoil_lsss}.

For any Kauffman state $D_s, s\in S(D)$, each component $\gamma\in\Gamma(s)$ of it contains an even number of cut loci by Corollary~\ref{cor:parity_cocycle_property}. The cut loci splits the components into edges.

An {\em enhanced Kauffman state} is an assingment a label $1,X$ or $-X$ to each edge of the components of $D_s$ such that any labels $u,v$ which belong to the adjacent edges of one component satisfy the condition $v=\tau(u)$, where $\tau(1)=1$, $\tau(\pm X)=\mp X$, see Fig.~\ref{fig:enhanced_state_involution_rule}.

\begin{figure}[h]
\centering\includegraphics[width=0.2\textwidth]{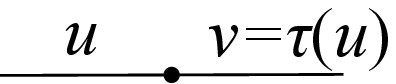}
\caption{Labels of adjacent edges in an enhance Kauffman state}\label{fig:enhanced_state_involution_rule}
\end{figure}

\begin{figure}[h]
\centering\includegraphics[width=0.6\textwidth]{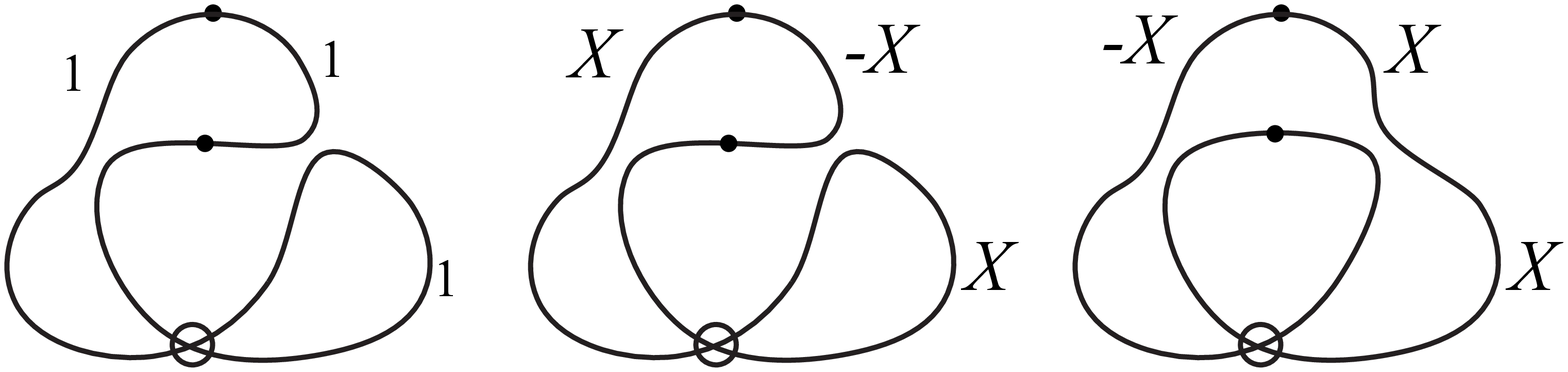}
\caption{Enhanced Kauffman states of the virtual trefoil}\label{fig:virtual_trefoil_enhanced_state}
\end{figure}

Let $V(s)$ be the free abelian group generated by the enhanced Kauffman states corresponding to the state $D_s$. (In fact, we identify some enhanced states. If two enhanced states $x$ and $y$ have the same labels up to sign, we identify $y$ with $(-1)^kx$ where $k$ is the number of components where the labels of $x$ and $y$ have different signs.)

The Khovanov chain space is the sum $CKh(D)=\bigoplus_{s\in S(D)} V(s)$. It has the {\em homological grading} $h$: for any $x\in V(s)$ its grading is $h(x)=|s|=\sum_{c\in\X(D)}s(c)$.

Fix an arbitrary order $\sigma_s\colon \{1,2,\dots, |\Gamma(s)|\}\to \Gamma(s)$ of the components for all Kauffman states $D_s$, $s\in S(D)$.

Choose a point $\star_s(\gamma)$ different from the cut loci on each component $\gamma\in\Gamma(s)$ of each Kauffman state $D_s$, $s\in S(D)$, see Fig.~\ref{fig:stars_and_order}.

\begin{figure}[h]
\centering\includegraphics[width=0.2\textwidth]{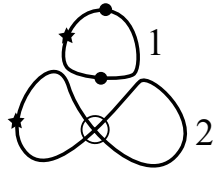}
\caption{Star points and an order of components}\label{fig:stars_and_order}
\end{figure}

The orders $\sigma_s$ and the star points $\star_s(\gamma)$ allow us to identify the module $V(s)$ with the tensor product $V^{\otimes |\Gamma(s)|}=V_{\sigma_s,\star_s}$. Any enhanced Kauffman state $x$ defines the element $x_1\otimes x_2\otimes\cdots\otimes x_{|\Gamma(s)|}\in V^{\otimes |\Gamma(s)|}$ where
$x_i$ is the label of the edge the point $\star_s({\sigma_s(i)})$ belongs to.

Let us now define the differentials of the Khovanov complex.

The Kauffman states of the diagram $D$ form a $|\X(D)|$-dimensional cube (the states are the vertices of this cube) and the differentials correspond to the edges of the cube. Let $s\to s'$ be an edge of the state cube. Then there exists a crossing $c\in\X(D)$ such that $s(c)=0, s'(c)=1$ and $s(c')=s'(c')$ for any $c'\in\X(D)$, $c'\ne c$. Hence, the edge $s\to s'$ corresponds to the change of smoothing in the crossing $c$.

There can be one of the three cases:
\begin{enumerate}
\item a component $\gamma\in\Gamma(s)$ splits into two components $\gamma'_1$ and $\gamma'_2\in\Gamma(s')$ (Fig.~\ref{fig:surgeries1to2} left);
\item two components $\gamma_1,\gamma_2\in\Gamma(s)$ merge into one component $\gamma'\in\Gamma(s')$ (Fig.~\ref{fig:surgeries1to2} right);
\item a component $\gamma\in\Gamma(s)$ transforms to a component $\gamma'\in\Gamma(s')$ (Fig.~\ref{fig:surgeries1to1}).
\end{enumerate}

\begin{figure}[h]
\centering\includegraphics[width=0.25\textwidth]{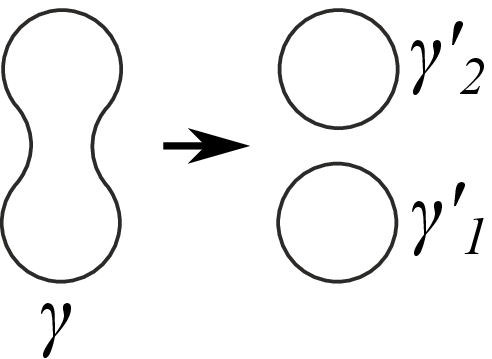}\qquad\qquad
\includegraphics[width=0.4\textwidth]{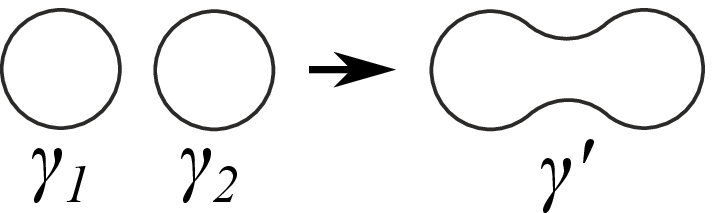}
\caption{Splitting and merging of components}\label{fig:surgeries1to2}
\end{figure}

\begin{figure}[h]
\centering\includegraphics[width=0.3\textwidth]{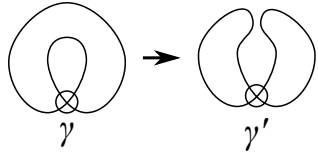}
\caption{Single circle surgery}\label{fig:surgeries1to1}
\end{figure}

The transformed components pass by the crossing $c$, so one can mark a point which corresponds to the crossing on each of these components.

We assign a map
\begin{gather*}
\partial_{s\to s'}\colon V_{\sigma_s,\star_s}\to V_{\sigma_{s'},\star_{s'}},\quad \partial_{s\to s'}=\sign(s,s')\partial^0_{s\to s'},\\
\partial^0_{s\to s'}(x_1\otimes\cdots\otimes x_{|\Gamma(s)|})=x'_1\otimes\cdots\otimes x'_{|\Gamma(s')|}.
\end{gather*}

to these transformations. The $\sign(s,s')$ will be defined later.

In the first case we define the unsigned differential $\partial^0_{s\to s'}$ as follows
\begin{equation}\label{eq:khovanov_differential_split}
x'_i=\left\{\begin{array}{cl}
\tau^{l_\lambda\left(\star_s(\sigma_{s'}(i)),\star_{s'}(\sigma_{s'}(i))\right)}(x_{\sigma_s^{-1}\circ\sigma_{s'}(i)}), & \sigma_{s'}(i)\ne\gamma'_1,\gamma'_2,\\
\tau^{l_\lambda(c,\star_{s'}(\gamma'_1))}(\tau^{l_\lambda(c,\star_s(\gamma))}(x_{\sigma_s^{-1}(\gamma)})^{(1)}), & \sigma_{s'}(i)=\gamma'_1,\\
\tau^{l_\lambda(c,\star_{s'}(\gamma'_2))}(\tau^{l_\lambda(c,\star_s(\gamma))}(x_{\sigma_s^{-1}(\gamma)})^{(2)}), & \sigma_{s'}(i)=\gamma'_2.
\end{array}\right.
\end{equation}
Here we identify any components of $D_s$ which doesn't pass by the crossing $c$ with the correspondent component in $D_{s'}$. We use the Sweedler notation $\Delta(x)=x^{(1)}\otimes x^{(2)}$ here~\cite{Sweedler}.

For any two points $p,q$ in one component of a Kauffman state, $l_\lambda(p,q)$ is the number of cut loci on an arc with the ends $p$ and $q$, see Fig.~\ref{fig:lpq}. According to the rule in Fig.~\ref{fig:enhanced_state_involution_rule}, the labels at the points $p$ and $q$ of the component containing them in an enhanced Kauffman state, are connected by the map $\tau^{l(p,q)}$.

\begin{figure}[h]
\centering\includegraphics[width=0.25\textwidth]{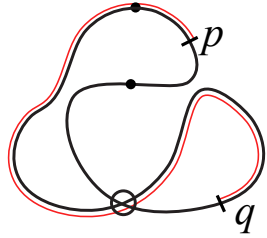}
\caption{The number $l_\lambda(p,q)=1$ does not depend on the arc connecting $p$ and $q$}\label{fig:lpq}
\end{figure}

In the second case
\begin{equation}\label{eq:khovanov_differential_merge}
x'_i=\left\{\begin{array}{cl}
\tau^{l_\lambda\left(\star_s(\sigma_{s'}(i)),\star_{s'}(\sigma_{s'}(i))\right)}(x_{\sigma_s^{-1}\circ\sigma_{s'}(i)}), & \sigma_{s'}(i)\ne\gamma',\\
\tau^{l_\lambda(c,\star_{s'}(\gamma'))}\left(\tau^{l_\lambda(c,\star_s(\gamma_1))}(x_{\sigma_s^{-1}(\gamma_1)})\cdot \tau^{l_\lambda(c,\star_s(\gamma_2))}(x_{\sigma_s^{-1}(\gamma_2)})\right), & \sigma_{s'}(i)=\gamma'.
\end{array}\right.
\end{equation}

In the third case we set $\partial_{s\to s'}=\partial^0_{s\to s'}=0$.

The differential of the Khovanov complex is the sum $d=\sum_{s\to s'} \partial_{s\to s'}$.

\begin{remark}
The formula~\eqref{eq:khovanov_differential_split} means the following. Given the label $x^\star_\gamma$ of the component $\gamma\in\Gamma(s)$ we determine the label of the component near the crossing $c$ where the splitting occurs, by applying $l_\lambda(c,\star_s(\gamma))$ times the involution $\tau$: $x^c_\gamma=\tau^{l_\lambda(c,\star_s(\gamma))}(x^\star_\gamma)$. Then we apply the comultiplication of the Frobenius algebra to $x^c_\gamma$ and get the labels of the split components near the crossing $c$: $\Delta(x^c_\gamma)=x^c_{\gamma'_1}\otimes x^c_{\gamma'_2}$. Finally, we find the labels of the components $\gamma'_1$ and $\gamma'_2$ at the star points: $x^\star_{\gamma'_1}=\tau^{l_\lambda(c,\star_s(\gamma'_1))}(x^\star_{\gamma'_1})$, $x^\star_{\gamma'_2}=\tau^{l_\lambda(c,\star_s(\gamma'_2))}(x^\star_{\gamma'_2})$.

We also rearrange the indices of the components according to the orders $\sigma_s$ and $\sigma_{s'}$.
\begin{figure}[h]
\centering\includegraphics[width=0.5\textwidth]{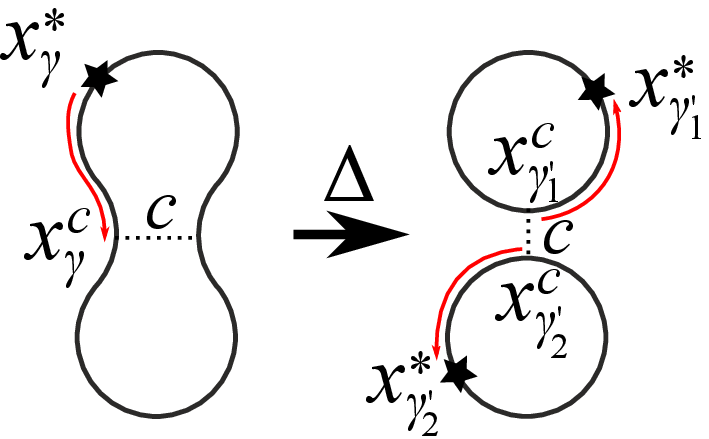}
\caption{}\label{fig:surgery1to2_stars}
\end{figure}
\end{remark}

Let us define the signs of the differential.

At each crossing we choose one of the incoming edges in the source-sink orientation as shown in Fig.~\ref{fig:source_sink_orders}. We mark the distinguished edge with a box. The choice of the edges for all classical crossings will be denoted by $o=\{o_c\}_{c\in\X(D)}$ and called an {\em oriented direction system} on the LSSS $\lambda$.

\begin{figure}[h]
\centering\includegraphics[width=0.4\textwidth]{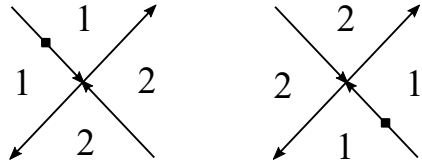}
\caption{Directions of a source-sink orientation}\label{fig:source_sink_orders}
\end{figure}

If $\lambda$ is a canonical LSSS of the oriented diagram $D$, we can consider the {\em canonical oriented direction system} choosing the incoming edges as shown in Fig.~\ref{fig:canonical_order}.

\begin{figure}[h]
\centering\includegraphics[width=0.4\textwidth]{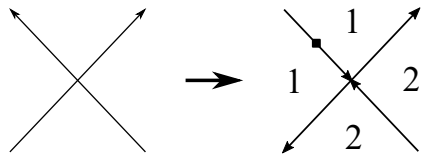}
\caption{Canonical oriented direction system}\label{fig:canonical_order}
\end{figure}

The oriented direction $o_c$ in a crossing $c$ defines a local ordering of the components in Kauffman states which pass by $c$, see Fig.~\ref{fig:source_sink_smoothing_order}. We can consider the local order as a family of maps $o^s_c\colon\{1,2\}\to\Gamma(s)$, $c\in\X(D)$, $x\in S(D)$ (the images of $1$ and $2$ may coincide).

\begin{figure}[h]
\centering\includegraphics[width=0.6\textwidth]{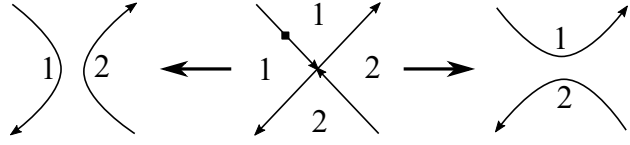}
\caption{Local ordering of component induced by an oriented direction}\label{fig:source_sink_smoothing_order}
\end{figure}

For any crossing $c\in\X(D)$, state $s\in S(D)$, global order $\sigma_s$ and oriented direction system $o$ we define two new orders:
\begin{itemize}
\item an order $\sigma_s\setminus c$ on the set $\Gamma(s)\setminus c=\Gamma(s)\setminus\im o^s_c$ consisting of the components in $D_s$ which don't pass by the crossing $c$. The order $\sigma_s\setminus c$ is the restriction of $\sigma_s$ to this set. Formally, let
    $j_1=\min\sigma_s^{-1}(\im o^s_c)$ and $j_2=\max\sigma_s^{-1}(\im o^s_c)$ be the indices of the component which pass by $c$ in the order $\sigma_s$. Then $j_1\le j_2$, they may coincide. We define
\begin{equation}\label{eq:sigma-c}
(\sigma_s\setminus c)(i)=\left\{\begin{array}{cl}
\sigma_s(i), & i<j_1,\\
\sigma_s(i+1), & j_1\le i<j_2-1,\\
\sigma_s(i+2), & i>=j_2-1.
\end{array}\right.
\end{equation}

\item an order $\sigma_s\triangleleft(o,c)$ on the set $\Gamma(s)$. Informally, we start the numbering of the components of $D_s$ with those that pass the crossing $c$, and enumerate them in the order $o_c^s$. The other components are ordered according $\sigma_s$. Let us give the explicit formulas. Let $n_s(c)$ be the number of component of $D_s$ which pass by the crossing $c$. Then
\begin{equation}\label{eq:sigma*oc}
(\sigma_s\triangleleft(o,c))(i)=\left\{\begin{array}{cl}
o_c^s(i), & i\le n_s(c),\\
(\sigma_s\setminus c)(i-n_s(c)), & i>n_s(c).
\end{array}\right.
\end{equation}
\end{itemize}

Given two orders $\sigma_1,\sigma_2\colon\{1,\dots,|Z|\}\to Z$ on a finite set $Z$, we define the sign $\epsilon(\sigma_1,\sigma_2)$ as the sign of the permutation $\sigma_1^{-1}\circ\sigma_2$ on the set $\{1,\dots,|Z|\}$.

Now we can define $\sign(s,s')$ of the differential $\partial_{s\to s'}$ as follows

\begin{equation}\label{eq:differential_sign}
\sign(s,s')=\epsilon(\sigma_s,\sigma_s\triangleleft(o,c))\cdot\epsilon(\sigma_s\setminus c,\sigma_{s'}\setminus c)\cdot\epsilon(\sigma_{s'},\sigma_{s'}\triangleleft(o,c)).
\end{equation}

\begin{example}
Consider a Hopf link diagram $D$ and choose a local source-sink structure $\lambda$ and an oriented direction system $o$ as shown in Fig.~\ref{fig:hopf_link}. Choose a global order $\{\sigma_s\}_{s\in S(D)}$ and star points $\{\star_s\}_{s\in S(D)}$ as shown in Fig.~\ref{fig:hopf_link_cube}.

\begin{figure}[h]
\centering\includegraphics[width=0.6\textwidth]{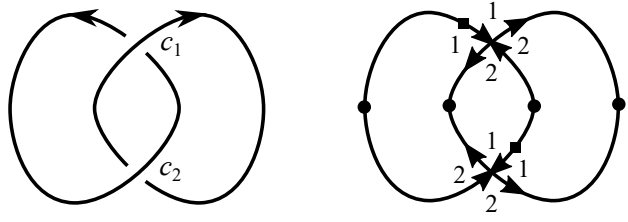}
\caption{Hopf link diagram and a LSSS and an oriented direction system on it}\label{fig:hopf_link}
\end{figure}

\begin{figure}[h]
\centering\includegraphics[width=0.6\textwidth]{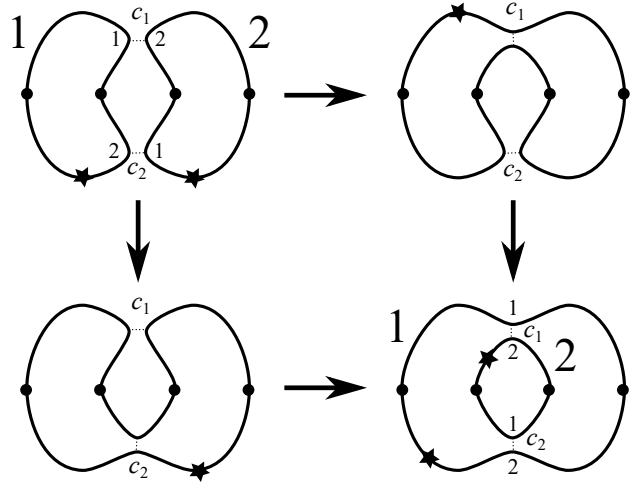}
\caption{Hopf link cube of states with a global order and star points}\label{fig:hopf_link_cube}
\end{figure}

Then $V(0,0)\simeq V(1,1)\simeq V\otimes V$, $V(0,1)\simeq V(1,0)\simeq V$.

Let us calculate the differential $\partial_{(0,0)\to (1,0)}$. It corresponds to the merging of two components in the crossing $c_1$. We have
$$
l_\lambda(c_1,\star_{(0,0)}(1))=l_\lambda(c_1,\star_{(0,0)}(2))=1,\quad l_\lambda(c_1,\star_{(1,0)}(1))=0.
$$
Hence, $\partial^0_{(0,0)\to (1,0)}\colon V\otimes V\to V$ is equal to $m\circ(\tau\otimes\tau)$.

Identifying the components with their numbers in the global order $\sigma_s$ we obtain  $\sigma_{(0,0)}\setminus c_1 =\emptyset$ , $\sigma_{(1,0)}\setminus c_1 =\emptyset$, $\sigma_{(0,0)}\triangleleft(o,c_1) =\sigma_{(0,0)}=\{1,2\}$, $\sigma_{(1,0)}\triangleleft(o,c_1) =\sigma_{(1,0)}=\{1\}$. Hence,
$$\epsilon(\sigma_{(0,0)},\sigma_{(0,0)}\triangleleft(o,c))=\epsilon(\sigma_{(1,0)},\sigma_{(1,0)}\triangleleft(o,c))= \epsilon(\sigma_{(0,0)}\setminus c_1,\sigma_{(1,0)}\setminus c_1)=1.
$$

Thus $\sign((0,0), (1,0))=1$ and  $\partial_{(0,0)\to (1,0)}= \partial^0_{(0,0)\to (1,0)}=m\circ(\tau\otimes\tau)$.

After calculating the other differential we have the following diagram.
$$
\xymatrix@C+3pc{V\otimes V \ar[r]^{m\circ(\tau\otimes\tau)}\ar[d]_{-m} & V\ar[d]^{-(\id\otimes\tau)\circ\Delta\circ\tau}\\
  V\ar[r]^{(\tau\otimes\id)\circ\Delta\circ\tau} & V\otimes V}
$$
The diagram is anticommutative and the homology $KH^*(D)$ of the complex is equal to $KH^0(D)=\Z^2$, $KH^1(D)=0$, $KH^2(D)=\Z^2$.
\end{example}

%
%

\begin{theorem}\label{thm:Khovanov_complex}
Let $D$ be a virtual link diagram. Then
\begin{enumerate}
\item $(CKh(D),d)$ is a chain complex, i.e. $d^2=0$;
\item the homology of the complex $(CKh(D),d)$ does not depend on the choice of auxiliary structures;
\item the homology of the complex $(CKh(D),d)$ coincides with the Khovanov homology of virtual links defined in~\cite{DKK2017,Manturov2007,Manturov2007a}.
\end{enumerate}
\end{theorem}

\begin{proof}
Firstly, we show that the choice of auxiliary structures does not change the homology. We construct isomorphisms between the complexes with different auxiliary structures.

1. Independence on the star points. Let $\tilde\star_s$ be a new set of star points in the Kauffman states $D_s, s\in S(D),$ and $\widetilde{CKh}(D)$ be the Khovanov complex corresponding to the new star points. The complex $\widetilde{CKh}(D)$  is isomorphic to the original complex $\widetilde{CKh}(D)$ via the map $\phi_s\colon V_{\sigma_s,\star_s}\to V_{\sigma_s,\tilde\star_s}$,
$$
\phi_s\left(\bigotimes_{i=1}^{|\Gamma(s')|}x_i\right)=\bigotimes_{i=1}^{|\Gamma(s')|} \tau^{l_\lambda(\star_s(\sigma_s(i)),\tilde\star_s(\sigma_s(i)))}(x_i).
$$

We should check the equality $\tilde\partial_{s\to s'}=\phi_{s'}\circ\partial_{s\to s'}\circ\phi_{s}^{-1}$. It follows from the equalities \begin{gather*}
l_\lambda(c,\tilde\star_s(\gamma))=l_\lambda(c,\star_s(\gamma))+l_\lambda(\star_s(\gamma),\tilde\star_s(\gamma)),\\
l_\lambda(\tilde\star_s(\gamma),\tilde\star_{s'}(\gamma))=l_\lambda(\star_s(\gamma),\tilde\star_s(\gamma))
+l_\lambda(\star_s(\gamma),\star_{s'}(\gamma))+l_\lambda(\star_{s'}(\gamma),\tilde\star_{s'}(\gamma)).
\end{gather*}
for any crossing $c\in\X(C)$ and any component $\gamma\in\Gamma(s)$ which passes by $c$ in the first equation, and doesn't pass in the second.


2. Independence on the global order. Let $\{\tilde\sigma_s\}_{s\in S(D)}$ be a new global order and $\widetilde{CKh}(D)=\bigoplus_s V_{\tilde\sigma_s,\star_s}$ be the Khovanov complex corresponding to the new order.  The isomorphism $\phi_s\colon V_{\sigma_s,\star_s}\to V_{\tilde\sigma_s,\star_s}$ is defined by the formula
$$
\phi_s(x_1\otimes\dots x_{|\Gamma(s)|})= \epsilon(\sigma_s,\tilde\sigma_s)x_{\sigma_s^{-1}\circ\tilde\sigma_s(1)}\otimes\dots\otimes
x_{\sigma_s^{-1}\circ\tilde\sigma_s(|\Gamma(s)|)}.
$$
We can write $\phi_s=\epsilon(\sigma_s,\tilde\sigma_s)\phi^0_s$. Since $\phi^0_s$ is a renumbering of multipliers in the product $V_{\sigma_s,\star_s}=V^{\otimes |\Gamma(s)|}$, we have $\tilde\partial^0_{s\to s'}=\phi^0_{s'}\circ\partial^0_{s\to s'}\circ(\phi_s^0)^{-1}$.

Let us check the sign. We have
\begin{gather*}
\epsilon(\tilde\sigma_s,\tilde\sigma_s\triangleleft(o,c))=\epsilon(\sigma_s,\tilde\sigma_s)\cdot
\epsilon(\sigma_s,\sigma_s\triangleleft(o,c))\cdot\epsilon(\sigma_s\triangleleft(o,c),\tilde\sigma_s\triangleleft(o,c)),\\
\epsilon(\tilde\sigma_s\setminus c,\tilde\sigma_{s'}\setminus c)=\epsilon(\sigma_s\setminus c,\tilde\sigma_s\setminus c)\cdot
\epsilon(\sigma_s\setminus c,\sigma_{s'}\setminus c)\cdot\epsilon(\sigma_{s'}\setminus c,\tilde\sigma_{s'}\setminus c),\\
\epsilon(\tilde\sigma_{s'},\tilde\sigma_{s'}\triangleleft(o,c))=\epsilon(\sigma_{s'},\tilde\sigma_{s'})\cdot
\epsilon(\sigma_{s'},\sigma_{s'}\triangleleft(o,c))\cdot\epsilon(\sigma_{s'}\triangleleft(o,c),\tilde\sigma_{s'}\triangleleft(o,c)).
\end{gather*}
By equation~\eqref{eq:sigma*oc} we have
$$\epsilon(\sigma_s\triangleleft(o,c),\tilde\sigma_s\triangleleft(o,c))=\epsilon(\sigma_s\setminus c,\tilde\sigma_s\setminus c),\  \epsilon(\sigma_{s'}\triangleleft(o,c),\tilde\sigma_{s'}\triangleleft(o,c))=\epsilon(\sigma_{s'}\setminus c,\tilde\sigma_{s'}\setminus c).
$$
Thus,
\begin{multline*}
\widetilde{\sign}(s,s')=\epsilon(\tilde\sigma_s,\tilde\sigma_s\triangleleft(o,c))\cdot\epsilon(\tilde\sigma_s\setminus c,\tilde\sigma_{s'}\setminus c)\cdot\epsilon(\tilde\sigma_{s'},\tilde\sigma_{s'}\triangleleft(o,c))=\\
\epsilon(\sigma_s,\sigma_s\triangleleft(o,c))\cdot\epsilon(\sigma_s\setminus c,\sigma_{s'}\setminus c)\cdot\epsilon(\sigma_{s'},\sigma_{s'}\triangleleft(o,c))=\sign(s,s').
\end{multline*}

3. Independence on the oriented direction system. Let $\tilde o$ be a new oriented direction system which differs from $o$ in some crossing $c\in \X(D)$, and $\widetilde{CKh}(D)$ be the Khovanov complex corresponding to the new structure. The complex $\widetilde{CKh}(D)=\bigoplus_{s\in S(D)} V_{\sigma_s,\star_s}$ has the same chains as ${CKh}(D)$ but it has new differentials
$$
\tilde\partial_{s\to s'}=\widetilde{\sign}(s,s')\partial^0_{s\to s'}
$$
which differ from $\partial_{s\to s'}$ by some signs.

Define isomorphisms $\phi_s\colon V_{\sigma_s,\star_s}\to V_{\sigma_s,\star_s}, s\in S(D),$ as $\phi_s=(-1)^{s(c)}\id_{V_{\sigma_s,\star_s}}$.
We recall that any state $s$ is a map from the set of crossings $\X(C)$ to $\{0,1\}$.

We need to check that for any surgery $s\to s'$ we have
\begin{equation}\label{eq:direction_change_sign}
\widetilde{\sign}(s,s')=(-1)^{s(c)}\sign(s,s')(-1)^{s'(c)}.
\end{equation}

Let $s\to s'$ be a surgery at some crossing $c'\ne c$. Then $s(c)=s'(c)$. The local order $\tilde o^s_{c'}$ coincides with $o^s_{c'}$, and  $\tilde o^{s'}_{c'}$ coincides with $o^{s'}_{c'}$. Hence, $\widetilde{\sign}(s,s')=\sign(s,s')$ and the equality~\eqref{eq:direction_change_sign} holds.

Let $s\to s'$ be a splitting at the crossing $c$. Then $s(c)\ne s'(c)$ and $(-1)^{s'(c)}=-(-1)^{s(c)}$. The local orders $\tilde o^s_{c}$ $o^s_{c}$ coincide since only one component of $D_s$ passes by $c$, and the orders $\tilde o^{s'}_{c}$ and  $o^{s'}_{c}$ are opposite. Then
$\epsilon(\sigma_s,\sigma_s\triangleleft(\tilde o,c))=\epsilon(\sigma_s,\sigma_s\triangleleft(o,c))$ and
\begin{multline*}
\epsilon(\sigma_{s'},\sigma_{s'}\triangleleft(\tilde o,c))=\epsilon(\sigma_{s'},\sigma_{s'}\triangleleft(o,c))\cdot\epsilon(\sigma_{s'}\triangleleft(o,c),\sigma_{s'}\triangleleft(\tilde o,c))=\\
\epsilon(\sigma_{s'},\sigma_{s'}\triangleleft(o,c))\cdot\epsilon(o^{s'}_{c},\tilde o^{s'}_{c})=-\epsilon(\sigma_{s'},\sigma_{s'}\triangleleft(o,c)).
\end{multline*}
Hence,
\begin{multline*}
\widetilde{\sign}(s,s')=\epsilon(\sigma_s,\sigma_s\triangleleft(\tilde o,c))\cdot\epsilon(\sigma_s\setminus c,\sigma_{s'}\setminus c)\cdot\epsilon(\sigma_{s'},\sigma_{s'}\triangleleft(\tilde o,c))=\\
-\epsilon(\sigma_s,\sigma_s\triangleleft(o,c))\cdot\epsilon(\sigma_s\setminus c,\sigma_{s'}\setminus c)\cdot\epsilon(\sigma_{s'},\sigma_{s'}\triangleleft(o,c))=-\sign(s,s'),
\end{multline*}
and the equality~\eqref{eq:direction_change_sign} is valid in this case.

The case of a merging of two components at the crossing $c$ can be considered analogously.

4. Independence on local source-sink structure. Let $\tilde\lambda$ is a new LSSS which differs from $\lambda$ in some crossing $c$.
We define the oriented direction system for $\tilde\lambda$ in the crossing $c$ as shown in Fig.~\ref{fig:source_sink_order_switch}.

\begin{figure}[h]
\centering\includegraphics[width=0.5\textwidth]{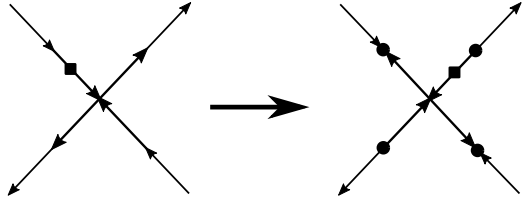}
\caption{Change of oriented direction with local source-sink orientation}\label{fig:source_sink_order_switch}
\end{figure}

Let $\widetilde{CKh}(D)=\bigoplus_{s\in S(D)} V_{\sigma_s,\star_s}$ be the Khovanov complex corresponding to the new structure. It has the same cochains as the complex ${CKh}(D)$ but different differentials $\tilde\partial_{s\to s'}$. Then we define an isomorphism $\phi\colon CKh(D)\to\tilde CKh(D)$ as follows
\begin{equation*}
\phi_s(x)=(-1)^{\mu s(c)}x
\end{equation*}
where $\mu=0$ if the sink edges of $\lambda$ were the overcrossing arc in $c$, and $\mu=1$ if the sink edges of $\lambda$ were the undercrossing arc in $c$, see Fig.~\ref{fig:source_sink_order_change0},\ref{fig:source_sink_order_change1}.

\begin{figure}[h]
\centering\includegraphics[width=0.5\textwidth]{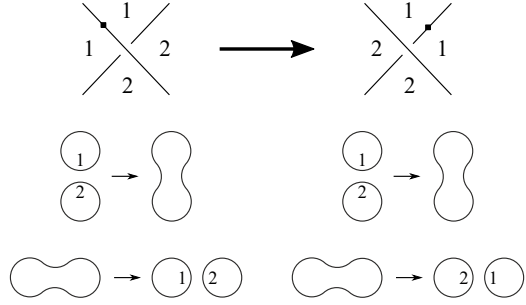}
\caption{Case $\mu=0$}\label{fig:source_sink_order_change0}
\end{figure}
\begin{figure}[h]
\centering\includegraphics[width=0.5\textwidth]{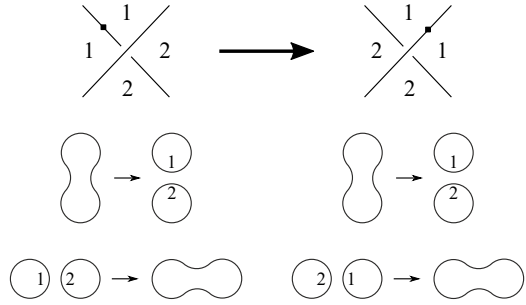}
\caption{Case $\mu=1$}\label{fig:source_sink_order_change1}
\end{figure}

Let $s\to s'$ be a surgery at some crossing $c'\ne c$. Then $s(c)=s'(c)$, hence $(-1)^{\mu s(c)}=(-1)^{\mu s'(c)}$. The LSSS $\tilde\lambda$ adds a cut locus to each edge incident to the crossing $c$. After smoothing these four cul loci splits in pairs. Assuming the star points lie far from $c$, we have
\begin{gather*}
l_{\tilde\lambda}(c',\star_s(\gamma))=l_\lambda(c',\star_s(\gamma))\pmod 2,\\
l_{\tilde\lambda}(\star_s(\gamma),\star_{s'}(\gamma))=l_{\lambda}(\star_s(\gamma),\star_{s'}(\gamma))\pmod 2
\end{gather*}
for any component $\gamma$ passing by $c'$ in the first equation, and not passing by $c'$ in the second equation.

Therefore, $\tilde\partial^0_{s\to s'}=\partial^0_{s\to s'}$ and $\tilde\partial_{s\to s'}=\partial_{s\to s'}$.

Let $s\to s'$ be a surgery at the crossing $c$. Then $s(c)\ne s'(c)$. The LSSS $\tilde\lambda$ adds a cut locus to each edge incident to the crossing $c$, so
$$
l_{\tilde\lambda}(c,\star_s(\gamma))=l_\lambda(c,\star_s(\gamma))+1,\quad l_{\tilde\lambda}(c,\star_{s'}(\gamma'))=l_\lambda(c,\star_{s'}(\gamma'))+1
$$
for any components $\gamma\in\Gamma(s)$ and $\gamma'\in\Gamma(s')$ which pass by the crossing $c$. Thanks to the identities
$$\tau\circ m\circ(\tau\otimes\tau)=m,\quad (\tau\otimes\tau)\circ\Delta\circ\tau=-\Delta,$$
we have $\tilde\partial^0_{s\to s'}=\partial^0_{s\to s'}$ for a merging in the crossing $c$, and $\tilde\partial^0_{s\to s'}=-\partial^0_{s\to s'}$ for a splitting in $c$.

Let us check the signs. Assume that the chosen sink edge in the crossing $c$ belongs to the over-arc (case $\mu=0$). Let the surgery $s\to s'$ be a merging. Then the local orders $o^s_c$ and $\tilde o^s_c$ coincide (see Fig.~\ref{fig:source_sink_order_change0}). Hence, $\widetilde{\sign}(s,s')=\sign(s,s')$ and $\tilde\partial_{s\to s'}=\partial_{s\to s'}$.

If the surgery $s\to s'$ is a splitting then the local orders $o^s_c$ and $\tilde o^s_c$ are opposite. Hence, $\widetilde{\sign}(s,s')=-\sign(s,s')$ and
$$\tilde\partial_{s\to s'}=\widetilde{\sign}(s,s')\tilde\partial^0_{s\to s'}=(-\sign(s,s'))(-\partial^0_{s\to s'})=\partial_{s\to s'}.
$$

When $\mu=1$ (see Fig.~\ref{fig:source_sink_order_change1}) we have $\tilde\partial_{s\to s'}=-\partial_{s\to s'}$ in all cases, so
$$
\phi_{s'}\partial_{s\to s'}\phi^{-1}_s=(-1)^{\mu s(c)}(-1)^{\mu s'(c)}(-\tilde\partial_{s\to s'})=\tilde\partial_{s\to s'}.
$$

Thus, the complex up to isomorphism does not depend on the auxiliary structures.

Now, assume that $D$ is oriented and $\lambda$ and $o$ are the canonical source-sink structure and the canonical oriented direction system. Then the complex $(CKh(D),d)$ coincides with the complex in paper~\cite{DKK2017}. Hence, the complex is well-defined and its homology is the Khovanov homology of the virtual link $D$.
\end{proof}

\begin{remark}
The constructions of Khovanov homology in~\cite{Manturov2007} and~\cite{DKK2017} rely on the canonical parity cocycles. The unoriented Khovanov complex~\cite{BKM2020} uses implicitly the virtual parity cocycle, which adds cut loci for every virtual crossing. Thus, our description of the Khovanov homology complex can be treated as a unification of those construction.
\end{remark}

\end{document}